\documentclass[a4paper,12pt,reqno]{amsart}

\usepackage{amsmath}
\usepackage{amsthm}
\usepackage{amssymb}
\usepackage{fullpage}

\newcommand{\N}{\mathbb{N}}
\newcommand{\R}{\mathbb{R}}
\newcommand{\C}{\mathbb{C}}

\newcommand{\B}{\mathbb{B}}
\newcommand{\I}{\,\mathrm{i}\,}
\renewcommand{\Re}{\operatorname{Re}}
\renewcommand{\Im}{\operatorname{Im}}
\renewcommand{\d}{\mathrm{d}}
\renewcommand{\O}{\mathcal O}
\newcommand{\Oo}{\mathcal{O}_{\mathrm o}}
\newcommand{\mc}{\mathcal}
\newcommand{\mb}{\mathbf}
\newcommand{\rg}{\operatorname{rg}}

\newtheorem{lemma}{Lemma}[section]
\newtheorem{proposition}[lemma]{Proposition}
\newtheorem{theorem}[lemma]{Theorem}
\newtheorem{corollary}[lemma]{Corollary}
\theoremstyle{remark}
\newtheorem{remark}[lemma]{Remark}

\theoremstyle{definition}
\newtheorem{definition}[lemma]{Definition}

\numberwithin{equation}{section}

\title[Strichartz estimates in similarity coordinates]{Strichartz estimates in similarity coordinates and stable blowup for
the critical wave equation}

\author{Roland Donninger}
\address{Rheinische Friedrich-Wilhelms-Universit\"at Bonn,
Mathematisches Institut, Endenicher Allee 60, D-53115 Bonn, Germany}
\address{Universit\"at Wien, Fakult\"at f\"ur Mathematik, Oskar-Morgenstern-Platz 1, A-1090 Vienna, Austria}
\email{donninge@math.uni-bonn.de}

\thanks{Roland Donninger is supported by the Alexander von 
Humboldt Foundation via a Sofja Kovalevskaja Award endowed by the German
Federal Ministry of Education and Research. Partial support by the Deutsche 
Forschungsgemeinschaft (DFG), CRC 1060 'The Mathematics
of Emergent Effects', is also gratefully acknowledged.}

\begin{document}
\maketitle

\begin{abstract}
We establish Strichartz estimates in similarity coordinates 
for the radial wave equation in three spatial dimensions with a (time-dependent)
self-similar potential.
As an application we consider the critical wave equation
and prove the asymptotic stability 
of the ODE blowup profile in the energy space.
\end{abstract}

\section{Introduction}
\noindent Strichartz estimates play a fundamental role in the study of nonlinear dispersive
wave equations.
Consider for instance the Cauchy problem for the energy-critical wave equation
\begin{equation}
\label{eq:main1}
\left \{
\begin{array}{l}
(\partial_t^2-\Delta_x)u(t,x)=u(t,x)^5 \\
u[0]=(f,g)
\end{array}
\right .
\end{equation} 
for $u: I\times \R^3\to \R$, $I\subset \R$ an interval, $0\in I$. 
We use the abbreviation 
\[ u[t]:=(u(t,\cdot),\partial_t u(t,\cdot)). \]
Eq.~\eqref{eq:main1} has the conserved energy
\begin{equation}
\label{eq:energy} E(u[t])=\tfrac12 \|u(t,\cdot)\|_{\dot H^1(\R^3)}^2
+\tfrac12 \|\partial_t u(t,\cdot)\|_{L^2(\R^3)}^2
-\tfrac16 \|u(t,\cdot)\|_{L^6(\R^3)}^6 
\end{equation}
and the scaling symmetry 
\begin{equation}
\label{eq:scaling}
u(t,x)\mapsto u_\lambda(t,x):=\lambda^{-\frac12}u(\tfrac{t}{\lambda},\tfrac{x}{\lambda}),\qquad 
\lambda>0.
\end{equation}
In order to have access to conservation laws like \eqref{eq:energy}, it is desirable to study Eq.~\eqref{eq:main1}
at the lowest regularity level possible. 
The energy \eqref{eq:energy} is invariant under the scaling \eqref{eq:scaling} and thus,
the \emph{energy space} 
$\dot H^1\times L^2(\R^3)$ is the natural Sobolev space for the Cauchy problem 
\eqref{eq:main1}.
In fact, it is not hard to see that $\dot H^1\times L^2(\R^3)$ is \emph{optimal} 
for local well-posedness in the
scale of homogeneous Sobolev spaces $\dot H^s\times \dot H^{s-1}(\R^3)$, see 
e.g.~\cite{LinSog95, Sog08}.
The appropriate weak formulation of Eq.~\eqref{eq:main1} in this context is provided by
\emph{Duhamel's formula}
\begin{equation}
\label{eq:Duhamel}
 u(t,\cdot)=\cos(t|\nabla|)f+\frac{\sin(t|\nabla|)}{|\nabla|}g
+\int_0^t \frac{\sin((t-s)|\nabla|)}{|\nabla|}u(s,\cdot)^5 \d s,
\end{equation}
which makes sense for energy initial data $(f,g)\in \dot H^1 \times L^2(\R^3)$.
The fundamental problem that occurs in constructing solutions to Eq.~\eqref{eq:Duhamel}
is the fact that the nonlinearity cannot be controlled by Sobolev embedding
since this would require $\dot H^1(\R^3)\hookrightarrow L^{10}(\R^3)$ which clearly fails.
The way out of this dilemma is provided by the Strichartz estimate
\[ \left \|\frac{\sin(t|\nabla|)}{|\nabla|}g\right \|_{L_t^5(\R)L^{10}(\R^3)}
\lesssim \|g\|_{L^2(\R^3)} \]
and variants thereof. Strichartz estimates encode the dispersive properties of
solutions to the free wave equation and it is precisely the 
dispersion which yields the decisive gain compared to
the ordinary Sobolev embedding.

Similar problems occur if one is to study the wave flow near a static solution 
of Eq.~\eqref{eq:main1} other than $0$.
For instance, Eq.~\eqref{eq:main1} has the solution 
\[ u(t,x)=W(x):=(1+\tfrac13|x|^2)^{-\frac12} \]
and studying the wave flow near $W$ at energy regularity requires appropriate
Strichartz estimates for the wave equation
\[ [\partial_t^2-\Delta_x+V(x)]u(t,x)=0 \]
with the potential $V(x)=-5W(x)^4$.
Strichartz estimates for wave equations with potentials are 
an active area of research, see e.g.~\cite{MarMetTatToh10, BurPlaStaTah04,
DAnPie05, DAnFan08, Bec14} for some recent results.

Another explicit solution of Eq.~\eqref{eq:main1} is the \emph{ODE blowup solution}
\[ u^T(t,x)=c_3(T-t)^{-\frac12},\qquad c_3:=(\tfrac34)^\frac14 \]
where $T>0$ is a free parameter.
The solution $u^T$ provides an explicit example of finite-time blowup.
In order to determine the role of $u^T$ for generic evolutions, it is necessary to investigate
its stability.
The study of the wave flow near $u^T$ at optimal regularity 
requires Strichartz estimates for wave equations with \emph{self-similar potentials}
of the form
\[ [\partial_t^2-\Delta_x+(T-t)^{-2}V(\tfrac{x}{T-t})]u(t,x)=0. \]
In this paper we prove such estimates for the first time.
By applying the machinery from \cite{Don11, DonSchAic11, DonSch12, DonSch14, Don14} we can 
then conclude the asymptotic stability of $u^T$ in the energy topology.
This is the first result on blowup stability for a wave equation at the optimal
regularity level. 
Our main result reads as follows (see Section \ref{sec:lightcone} below for the precise
solution concept we are using).

\begin{theorem}
\label{thm:main}
Let $(f,g)$ be radial.
There exist constants $M,\delta>0$ such that, if
\[ \|(f,g)-u^1[0]\|_{H^1 \times L^2(\B^3_{1+\delta})}\leq \tfrac{\delta}{M}, \]
then the blowup time $T=T_{f,g}(0)$ at the origin is in $[1-\delta,1+\delta]$ and the
corresponding solution $u$ of Eq.~\eqref{eq:main1} satisfies
\[ \int_0^T \left [\frac{\|u(t,\cdot)-u^T(t,\cdot)\|_{L^\infty(\B^3_{T-t})}}
{\|u^T(t,\cdot)\|_{L^\infty(\B^3_{T-t})}}\right ]^2 \frac{\d t}{T-t}\lesssim \delta^2. \]
\end{theorem}

\subsection{Rough solutions in lightcones}
\label{sec:lightcone}

Throughout the paper we restrict ourselves to radial data. Consequently, the
effective Cauchy problem we study reads
\begin{equation}
\label{eq:main}
 \left \{
\begin{array}{l}
(\partial_t^2-\partial_r^2-\tfrac{2}{r}\partial_r)u(t,r)=u(t,r)^5 \\
u[0]=(f,g)
\end{array}
\right. 
\end{equation}
where $r=|x|$.
For a detailed study of blowup it is necessary to localize the well-posedness theory
to lightcones. 
This is straightforward for classical solutions but not entirely trivial for energy-class
data.
Note that, by Sobolev embedding and H\"older's inequality, $f\in \dot H^1(\R^3)$
implies $f\in H^1_{\mathrm{loc}}(\R^3)$.
Consequently, for the Cauchy problem restricted to a lightcone, we may assume
that the data belong to $H^1\times L^2$.  
As discussed in \cite{KilStoVis14}, one may still use the Duhamel formula \eqref{eq:Duhamel}
combined with a cut-off technique and finite speed of propagation to define
energy-class solutions in lightcones.
A different approach is based on the introduction of similarity coordinates
\begin{equation}
\label{eq:simcoord}
 \tau:=-\log(T-t)+\log T,\qquad \rho:=\frac{r}{T-t},\qquad T>0 
 \end{equation}
and semigroup theory.
The coordinate transformation \eqref{eq:simcoord} maps the lightcone
\[ \Gamma_T:=\{(t,r): t \in [0,T), r\leq T-t\} \]
to the infinite cylinder $[0,\infty)\times \overline{\B^3}$.
If $u\in C^\infty(\Gamma_T)$ is a classical solution of Eq.~\eqref{eq:main}, we define
$\psi_T \in C^\infty([0,\infty)\times \overline{\B^3})$ by 
\begin{equation}
\label{eq:utopsiintro}
 u(t,r)=(T-t)^{-\frac{1}{2}}\psi_T(-\log(T-t)+\log T,\tfrac{r}{T-t}).
 \end{equation}
 Then Eq.~\eqref{eq:main} is equivalent to
\[ \left \{ \begin{array}{l}
\partial_\tau \Psi_T(\tau)=\tilde{\mb L}_0 \Psi_T(\tau)+\mb F(\Psi_T(\tau)) \\
\Psi_T(0)(\rho)=\left ( \begin{array}{c}
T^\frac{1}{2}f(T\rho) \\
 T^\frac{3}{2}g(T\rho) \end{array} \right )
 \end{array} 
\right .  
 \]
for 
\begin{equation}
\label{eq:Psiintro} \Psi_T(\tau)(\rho):=\left (\begin{array}{c}
\psi_T(\tau,\rho) \\
(\partial_\tau+\rho\partial_\rho+\frac{1}{2})
\psi_T(\tau,\rho) \end{array} \right ), 
\end{equation}
where $\tilde{\mb L}_0$ is a spatial differential operator and $\mb F$ denotes the nonlinearity.
From \cite{DonSch12}, see also Proposition \ref{prop:gen} below, it follows that 
(a closed extension of)
$\tilde{\mb L}_0$, augmented with a suitable domain, generates a semigroup $\mb S_0(\tau)$
on $H^1\times L^2(\B^3)$.
This leads to the following natural definition of energy-class solutions in lightcones.

\begin{definition}
We say that $u$ is an \emph{energy-class solution of Eq.~\eqref{eq:main} in the lightcone
$\Gamma_T$} if the corresponding $\Psi_T$, given by Eqs.~\eqref{eq:utopsiintro} 
and \eqref{eq:Psiintro},
belongs to $C([0,\infty),H^1\times L^2(\B^3))$ and satisfies
\[ \Psi_T(\tau)=\mb S_0(\tau)\Psi_T(0)+\int_0^\tau \mb S_0(\tau-\sigma)\mb F(\Psi_T(\sigma))\d\sigma \]
for all $\tau\geq 0$.
\end{definition}

The concept of energy-class solution in lightcones leads to the definition of the local blowup
time.

\begin{definition}
For given (radial) data $(f,g)\in \dot H^1\times L^2(\R^3)$ we say that 
$T\in A_{f,g}\subset \R_+$
if there exists an energy-class solution to Eq.~\eqref{eq:main} in the lightcone $\Gamma_T$.
We set 
\[ T_{f,g}(0):=\sup A_{f,g}\cup \{0\}. \] If $T_{f,g}(0)<\infty$, $T_{f,g}(0)$ is called
the \emph{blowup time (at the origin)}.
\end{definition}

\subsection{Related work}
Needless to say, the critical wave equation \eqref{eq:main1} attracted a lot of interest
in the recent past. One particularly intriguing feature is the existence of two competing
blowup mechanisms. In addition to the ODE blowup (type I blowup), there exists 
so-called type II blowup which is closely related to solitons 
and characterized by the boundedness of the energy norm.
In particular in the study of type II blowup,
there was spectacular progress in the last few years.
In their seminal work, Kenig and Merle established a blowup/scattering dichotomy 
\cite{KenMer08}, see also \cite{DuyMer08}.
The first construction of type II blowup solutions is due to Krieger, Schlag, and Tataru
\cite{KriSchTat09}, see \cite{DonHuaKriSch14, KriSch14} for further work in this direction.
An alternative approach was developed by Hillairet and Rapha\"el \cite{HilRap12}, cf.~also 
the recent work by Jendrej \cite{Jen15a, Jen15b}.
The author and Krieger constructed nondispersive solutions and solutions that blow up
in infinite time \cite{DonKri13}.
In a series of papers, Duyckaerts, Kenig, and Merle
obtained a complete classification of type II behavior
\cite{DuyKenMer11, DuyKenMer12a, DuyKenMer12b, DuyKenMer15a, DuyKenMer15b, DuyKenMer14b,
DuyKenMer15c}.
Krieger, Nakanishi, and Schlag established a number of fundamental results from the 
dynamical systems
point of view \cite{KriSch07, KriNakSch13a, KriNakSch13b, KriNakSch14}.

Concerning type I blowup, less is known. 
Bizo\'n, Chmaj, and Tabor \cite{BizChmTab04} gave strong numerical evidence that
blowup is generically of type I and described by the ODE profile.
The author and Sch\"orkhuber \cite{DonSch14, DonSch15a, DonSch15b} 
proved the asymptotic stability of the ODE blowup profile, but in the stronger
topology $H^2\times H^1$. We also mention the recent paper by Krieger and Wong \cite{KriWon14}
on continuation beyond type II singularities.
Unfortunately, the impressive machinery \cite{MerZaa03, MerZaa05, MerZaa07, 
MerZaa08, MerZaa12a, MerZaa12b} developed by Merle and Zaag for studying
type I blowup is confined to energy-subcritical equations and does not apply here.

\subsection{Outline of the proof}
The starting point for our earlier work \cite{DonSch14, DonSch15a, DonSch15b} was the following
simple observation. If $u: \R\times \R^3\to \R$ 
solves the \emph{free} wave equation 
\[ (\partial_t^2-\Delta_x)u(t,x)=0, \] 
then, by energy conservation,
the function $\psi_T$, defined by Eq.~\eqref{eq:utopsiintro},
satisfies 
\begin{align*} 1&\gtrsim \|u(t,\cdot)\|_{\dot H^2(\B^3_{T-t})}=
(T-t)^{-\frac{1}{2}}\|\psi_T(-\log(T-t)+\log T,\tfrac{\cdot}{T-t})\|_{\dot H^2(\B^3_{T-t})}\\
&=(T-t)^{-1}\|\psi_T(-\log(T-t)+\log T, \cdot)\|_{\dot H^2(\B^3)}.
 \end{align*}
 In other words, $\|\psi_T(\tau,\cdot)\|_{\dot H^2(\B^3)}\lesssim Te^{-\tau}$.
Consequently, for pure scaling reasons, one gets exponential decay in $H^2\times H^1$
for the free evolution
in similarity coordinates.
In the aforementioned references \cite{DonSch14, DonSch15a, DonSch15b} we were able
to propagate this decay to the nonlinear problem via a perturbative argument.
As the decay comes exclusively from scaling, it was not necessary to exploit
any \emph{dispersive} properties of the wave operator.
However, in order to see the scaling decay, one has to require the data to be in $H^2\times H^1$
which is far from optimal\footnote{We remark that \cite{DonSch14, DonSch15a, DonSch15b}
are mainly concerned with the energy-\emph{super}critical regime. Consequently,
this regularity issue is of minor importance there.} in view of the well-posedness theory for Eq.~\eqref{eq:main}.
Unfortunately, 
if one lowers the degree of regularity all the way down to the critical 
$\dot H^1\times L^2$,
one loses the decay from scaling and this makes the problem
much harder.

The absence of scaling decay necessitates the development of a completely different approach
which has to crucially exploit the dispersive behavior of the wave operator in similarity
coordinates.
On the technical level we accomplish this by proving Strichartz estimates for
the semigroup in question.
In what follows we give a more detailed outline of the paper.

\begin{itemize}
\item As explained above, the introduction of similarity coordinates \eqref{eq:simcoord}
leads to an evolution problem of the form
\begin{equation}
\label{eq:NLintro}
\partial_\tau \Psi_T(\tau)=\tilde{\mb L}_0 \Psi_T(\tau)+\mb F(\Psi_T(\tau)). 
\end{equation}
For brevity we drop the subscript $T$ and write $\Psi=\Psi_T$.
In this formulation, the 
ODE blowup solution $u^T$ corresponds to the constant function $(c_3, \frac{1}{2}c_3)$,
see Eqs.~\eqref{eq:utopsiintro} and \eqref{eq:Psiintro}.
Thus, we insert the ansatz $\Psi=(c_3,\frac{1}{2}c_3)+\Phi$ into Eq.~\eqref{eq:NLintro}
which yields
\[ \partial_\tau \Phi(\tau)=\tilde{\mb L}_0 \Phi(\tau)+\mb L'\Phi(\tau)+\mb N(\Phi(\tau)) \]
where the ``potential term'' $\mb L'\Phi(\tau)$ comes from 
the linearization of $\mb F$ at $(c_3,\frac{1}{2}c_3)$
and $\mb N(\Phi(\tau))$ is the nonlinear remainder.

\item Following \cite{DonSch12}, we prove that a closed extension of $\tilde{\mb L}_0+\mb L'$,
denoted by $\mb L$,
generates a semigroup $\mb S(\tau)$ on $\mc H:=H^1\times L^2(\B^3)$.
The generator $\mb L$ has precisely one unstable eigenvalue $\lambda=1$ and
$\sigma(\mb L)\backslash \{1\}\subset \{z\in \C: \Re z\leq 0\}$.
However, the eigenvalue $1$ does not indicate a ``real'' instability of the blowup solution
$u^T$ but is related to the time translation symmetry of the equation.
Furthermore, the Riesz projection $\mb P$ associated to the eigenvalue $1$ has rank one and
from semigroup theory we infer the bound $\|\mb S(\tau)(\mb I-\mb P)\|_{\mc H}
\leq C_\epsilon e^{\epsilon\tau}$ for any $\epsilon>0$.

\item By Laplace inversion, we obtain an explicit representation of 
$\mb S(\tau)(\mb I-\mb P)$ in terms of the resolvent of $\mb L$. 
Indeed, setting $\tilde{\mb f}=(\tilde f_1,\tilde f_2):=(\mb I-\mb P)\mb f$
for sufficiently regular $\mb f$, 
the first component of $\mb S(\tau)(\mb I-\mb P)\mb f$ is given by
\begin{equation}
\label{eq:Laplaceintro} [\mb S(\tau)\tilde{\mb f}]_1(\rho)=\frac{1}{2\pi \I}
\lim_{N\to\infty}\int_{\epsilon-\I N}^{\epsilon+\I N}e^{\lambda\tau}
\int_0^1 G(\rho,s;\lambda)F_\lambda(s)\d s \d\lambda 
\end{equation}
where $G$ is the Green function of the spectral ODE associated to $\mb L$
and 
\[ F_\lambda(s):=s\tilde f_1'(s)+(\lambda+\tfrac32)\tilde f_1(s)+\tilde f_2(s). \]
Eq.~\eqref{eq:Laplaceintro} holds for any $\epsilon>0$ and the goal is to 
take the limit $\epsilon\to 0$, i.e., to 
push the contour of integration to the imaginary axis.
This requires precise pointwise bounds on $G$ and one has to exploit oscillations.

\item We construct the Green function $G$ by a perturbative ODE analysis which yields the
representation $G(\rho,s;\lambda)=G_0(\rho,s;\lambda)+\tilde G(\rho,s;\lambda)$ 
where $G_0$ is the (explicitly known) Green
function of the free equation and the perturbing kernel $\tilde G$ has nice decay
properties as $|\Im\lambda|\to\infty$.
Consequently, Eq.~\eqref{eq:Laplaceintro} splits into a free part $[\mb S_0(\tau)\tilde{\mb f}]_1$
and a perturbation $T(\tau)\mb{\tilde f}$, where $\mb S_0(\tau)$ is the semigroup
generated by (a closed extension of) $\tilde{\mb L}_0$.

\item Next, we prove the Strichartz estimates
\[ \|[\mb S_0(\cdot)\tilde{\mb f}]_1\|_{L^p(\R_+)L^q(\B^3)}\lesssim \|\tilde{\mb f}\|_{\mc H},\qquad
\tfrac{1}{p}+\tfrac{3}{q}=\tfrac12 \]
in the range $p \in [2,\infty]$, $q\in [6,\infty]$.
This is done by employing the physical space representation of $\mb S_0(\tau)$ based on
d'Alembert's formula and an argument by Klainerman and Machedon \cite{KlaMac93}.
We remark that our Strichartz estimates include the endpoint $L^2(\R_+)L^\infty(\B^3)$ which
is crucial for the construction.
By delicate oscillatory integral estimates we prove the same Strichartz estimates for
the perturbation $T(\tau)\tilde {\mb f}$ which finally yields
\[ \|[\mb S(\cdot)(\mb I-\mb P)\mb f]_1\|_{L^p(\R_+)L^q(\B^3)}\lesssim \|\mb f\|_{\mc H} \]
for the above range of exponents $p,q$.
In a similar vein we improve the growth bound 
$\|\mb S(\tau)(\mb I-\mb P)\|_{\mc H}\leq C_\epsilon e^{\epsilon\tau}$ from semigroup theory
to $\|\mb S(\tau)(\mb I-\mb P)\|_{\mc H}\lesssim 1$.

\item The Strichartz estimates allow us to control the nonlinear terms in a way
similar to the standard local well-posedness theory
and we are able to run the program from \cite{DonSch12} to complete the proof
of Theorem \ref{thm:main}.

\end{itemize}

\subsection{Additional remarks}
In addition to the main result Theorem \ref{thm:main}, we hope that the Strichartz estimates 
in similarity
coordinates are of independent interest.
In this respect it is worth noting that our perturbative construction of the Green function is very robust and
does not use any specific properties of the potential. 
Thus, if one is able to derive the necessary spectral information, the proof of Strichartz
estimates along the lines of this paper works for essentially arbitrary potentials.
In view of the fact that the long-standing spectral issues related to blowup in 
supercritical wave maps and Yang-Mills models have recently been solved \cite{CosDonXia14, CosDonGloHua15, CosDonGlo16},
it is very likely that the techniques introduced in the present paper can also be applied
to these problems.

In this work we restrict ourselves to the radial case. There are two major problems one needs to address in order to remove the symmetry assumption. First, we use the $L^2L^\infty$ Strichartz endpoint to deal with the quadratic term in the nonlinearity. As is well known, this endpoint estimate fails outside of spherical symmetry \cite{KeeTao98}. This issue might be circumvented by using so-called reverse Strichartz estimates instead, cf.~\cite{Bec14}.
Second, we heavily rely on asymptotic ODE analysis to obtain a representation for the solution. In the nonsymmetric case, the corresponding elliptic equation is a PDE. 
However, since the solution one perturbs around is radial, one may use a decomposition in spherical harmonics to transform this PDE to a system of decoupled ODEs where a similar analysis as in this paper would apply. This approach was used in \cite{DonSch15a} to remove the symmetry assumption. Alternatively, one may try to adapt the robust methods developed by Tataru and collaborators \cite{MarMetTat08, Tat08, MetTat12} for dealing with variable coefficient equations.

\subsection{Notation}
Most of the notation we use is standard in the field or self-explanatory. 
We denote by $\B^d_R$ the open ball of radius $R>0$ in $\R^d$, centered at the origin.
For brevity we write $\B^d:=\B^d_1$ and $\R_+:=(0,\infty)$.
Bold letters denote $2$-component functions, e.g.~$\mb u=(u_1,u_2)$.
Throughout, we work with radial functions, i.e., $f(x)=\tilde f(|x|)$, 
and we identify $f$ with $\tilde f$. For Strichartz norms we use the notation
\[ \|u\|_{L^p(I)L^q(U)}=\|u(t,\cdot)\|_{L^p_t(I)L^q(U)}=\left (
\int_I \|u(t,\cdot)\|_{L^q(U)}^p \d t \right )^{1/p}. \]
For the Wronskian we use the convention $W(f,g):=fg'-f'g$.
Furthermore, we write $f(x)=O(g(x))$ if $f$ satisfies
$|f(x)|\lesssim |g(x)|$.
The notation $f(x)\sim g(x)$ as $x\to a$ means $\lim_{x\to a}\frac{f(x)}{g(x)}=1$.
The letter $C$ (possibly with subscripts to denote dependencies) 
stands for a positive constant that might change its value at each occurrence.
We also employ the ``Japanese bracket'' notation $\langle x\rangle:=\sqrt{1+|x|^2}$.

\section{Well-posedness of the Cauchy problem}

\subsection{Similarity coordinates}

Transforming Eq.~\eqref{eq:main} to similarity coordinates
\[ \tau=-\log(T-t)+\log T,\qquad \rho=\frac{r}{T-t} \] 
yields 
\begin{align*} 
\big [\partial_\tau^2&+2\partial_\tau+2\rho\partial_\tau\partial_\rho-(1-\rho^2)
\partial_\rho^2-\tfrac{2}{\rho}\partial_\rho 
+3\rho\partial_\rho+\tfrac{3}{2}(\tfrac{3}{2}-1)\big ]\psi(\tau,\rho)
=\psi(\tau,\rho)^5
\end{align*}
where 
\begin{equation}
\label{eq:utopsi}
 u(t,r)=(T-t)^{-\frac{1}{2}}\psi\big (-\log(T-t)+\log T,\tfrac{r}{T-t}\big ).
 \end{equation}
In these coordinates, the blowup solution $u^T$ reads
\[ \psi^T(\tau,\rho):=T^\frac{1}{2}e^{-\frac{1}{2}\tau}u^T(T-Te^{-\tau},Te^{-\tau}\rho)
=c_3.  \]
In order to obtain a first-order system in $\tau$, we introduce the variables
\begin{align}
\label{eq:psitopsi12}
\psi_1(\tau,\rho)&:=\psi(\tau,\rho) \nonumber \\
\psi_2(\tau,\rho)&:=\big [\partial_\tau+\rho\partial_\rho+\tfrac{1}{2}\big ]\psi(\tau,\rho).
\end{align}
This yields the system
\begin{align*}
\partial_\tau \psi_1&=-\rho\partial_\rho \psi_1-\tfrac{1}{2}\psi_1+\psi_2 \\
\partial_\tau \psi_2&=\partial_\rho^2 \psi_1+\tfrac{2}{\rho}\partial_\rho \psi_1
-\rho\partial_\rho \psi_2-\tfrac{3}{2}\psi_2+\psi_1^5 .
\end{align*}
The ODE blowup solution $u^T$ corresponds to the static solution
\[ \psi_1(\tau,\rho)=c_3,\qquad \psi_2(\tau,\rho)=\tfrac{1}{2}c_3. \]
We make the ansatz
\begin{equation}
\label{eq:ansatzphi}
 (\psi_1,\psi_2)=(c_3,\tfrac{1}{2}c_3)+(\phi_1,\phi_2). 
 \end{equation}
This yields the evolution equation
\begin{align}
\label{eq:phisys}
\partial_\tau \phi_1&=-\rho\partial_\rho \phi_1-\tfrac{1}{2}\phi_1+\phi_2 \nonumber \\
\partial_\tau \phi_2&=\partial_\rho^2 \phi_1+\tfrac{2}{\rho}\partial_\rho \phi_1
-\rho\partial_\rho \phi_2-\tfrac{3}{2}\phi_2+\tfrac{15}{4}\phi_1+N(\phi_1)
\end{align} 
where
\[ N(\phi_1):=10c_3^3 \phi_1^2+10 c_3^2\phi_1^3+5c_3 \phi_1^4+\phi_1^5. \]

\subsection{Semigroup theory}
Next, we develop the well-posedness theory for the linear Cauchy problem
associated to Eq.~\eqref{eq:phisys}.
In order to write Eq.~\eqref{eq:phisys} in more abstract form, we define the (formal) differential operators
\[ \tilde {\mb L}_0 \mb u(\rho):=\left ( \begin{array}{c}
-\rho u_1'(\rho)-\frac{1}{2} u_1(\rho)+u_2(\rho) \\
u_1''(\rho)+\frac{2}{\rho}u_1'(\rho)-\rho u_2'(\rho)-\frac{3}{2}u_2(\rho) 
\end{array} \right ) \]
and 
\[ \mb L' \mb u(\rho):=\left (\begin{array}{c}
0 \\ \frac{15}{4}u_1(\rho) \end{array} \right ), \]
acting on 2-component functions $\mb u=(u_1,u_2)$.
Furthermore, we set
\[ \mb N(\mb u):=\left (\begin{array}{c} 0 \\ N(u_1) \end{array} \right ). \]
Then we can write the system \eqref{eq:phisys} succinctly as
\begin{equation}
\label{eq:sysabs} \partial_\tau \Phi(\tau)=(\tilde{\mb L}_0+\mb L')\Phi(\tau) 
+\mb N(\Phi(\tau))
\end{equation}
for $\Phi(\tau)(\rho):=(\phi_1(\tau,\rho),\phi_2(\tau,\rho))$.

To study the evolution, we employ the semigroup machinery.
To this end, it is necessary to promote the formal differential operators $\tilde{\mb L}_0$
and $\mb L'$ to linear operators acting on a suitable Banach space.
We set 
\[ \mc H:=\{\mb f\in H^1\times L^2(\B^3): \mb f\mbox{ is radial}\} \] and write
\[ \|\mb f\|_{\mc H}^2:=\|f_1\|_{H^1(\B^3)}^2+\|f_2\|_{L^2(\B^3)}^2 \]
for $\mb f=(f_1,f_2)$.
Next, we augment the operator $\tilde{\mb L}_0$ with
the domain 
\[ \mc D(\tilde{\mb L}_0):=\{\mb u\in C^2\times C^1([0,1]): u_1'(0)=0\}. \]
In this way, $\tilde{\mb L}_0$ becomes a densely defined unbounded linear operator
on $\mc H$.
We claim that $\tilde{\mb L}_0$ has a closed extension, denoted by $\mb L_0$, which generates a semigroup.

\begin{proposition}
\label{prop:gen}
The operator $\tilde{\mb L}_0: \mc D(\tilde{\mb L}_0)\subset \mc H\to\mc H$ has a closed extension $\mb L_0$ that generates a strongly-continuous one-parameter semigroup
$\{\mb S_0(\tau): \tau\geq 0\}$ on $\mc H$ satisfying
\[ \|\mb S_0(\tau)\mb f\|_{\mc H}\lesssim \|\mb f\|_{\mc H} \]
for all $\tau\geq 0$ and all $\mb f\in \mc H$.
\end{proposition}

\begin{proof}
We define $\mb G: H^1\times L^2(\B^3)\to L^2(0,1)^2$ by
\begin{equation}
\label{eq:G} \mb G \mb f(\rho):=\left (\begin{array}{c}
\rho f_2(\rho) \\ \rho f_1'(\rho)+f_1(\rho) \end{array} \right ). 
\end{equation}
An integration by parts and the Sobolev embedding $H^1(\frac12,1)\hookrightarrow L^\infty(\frac12,1)$
show that
 $\|\mb G \mb f\|_{L^2(0,1)^2}\simeq \|\mb f\|_{H^1\times L^2(\B^3)}$ and the inverse of
 $\mb G$ is given by
 \[ \mb G^{-1}\mb f(\rho)=\frac{1}{\rho}\left ( \begin{array}{c}
 \int_0^\rho f_2(s)\d s \\
 f_1(\rho) \end{array} \right ). \]
 Consequently, $\mb G$ is a Banach space isomorphism.
 Furthermore, $\mb G$ maps $\mc D(\tilde{\mb L}_0)$ to 
 \[ \tilde{\mc D}:=\{\mb u\in C^1([0,1])^2: u_1(0)=0\}. \]
 For $\mb u\in \mb G(\mc D(\tilde{\mb L}_0))\subset \tilde{\mc D}$ we obtain
 \[ \tilde{\mb L}_0\mb G^{-1}\mb u(\rho)=\left (\begin{array}{c}
 \frac{1}{\rho}u_1(\rho)-u_2(\rho)+\frac{1}{2\rho}\int_0^\rho u_2(s)\d s \\
 -u_1'(\rho)-\frac{1}{2\rho}u_1(\rho)+\frac{1}{\rho}u_2'(\rho) \end{array} \right ) \]
 and thus,
 \begin{equation}
 \label{eq:GLG}
  \mb G\tilde{\mb L}_0\mb G^{-1}\mb u(\rho)=
 \left (\begin{array}{c}
 -\rho u_1'(\rho)-\frac12 u_1(\rho)+u_2'(\rho) \\
 u_1'(\rho)-\rho u_2'(\rho)-\frac12 u_2(\rho) \end{array} \right ). 
 \end{equation}
 The operator $\mb G\tilde{\mb L}_0\mb G^{-1}$ with domain
 $\tilde{\mc D}$ was studied in detail\footnote{One needs to set
 $p=5$ in \cite{DonSch12}. This might seem odd because \cite{DonSch12} is confined to
 the case $p\leq 3$. However, the linear theory in \cite{DonSch12} works for all $p>1$.} 
 in \cite{DonSch12}, see also \cite{Don10}. 
 In particular, Lemma 3.1 in \cite{DonSch12} shows that $\mb G\tilde{\mb L}_0\mb G^{-1}$ is closable and its closure generates a semigroup $\mb T_0(\tau)$ on $L^2(0,1)^2$ satisfying $\|\mb T_0(\tau)\mb f\|_{L^2(0,1)^2}\leq \|\mb f\|_{L^2(0,1)^2}$ for all $\tau\geq 0$ and $\mb f\in L^2(0,1)^2$.
 The claim now follows by setting $\mb S_0(\tau)=\mb G^{-1}\mb T_0(\tau)\mb G$.
\end{proof}

\subsection{Strichartz estimates for the free evolution}
In the following we prove Strichartz estimates for the semigroup $\mb S_0$.
We employ an argument by Klainerman and Machedon \cite{KlaMac93} which
is based on d'Alembert's formula and the $L^2$-boundedness of the Hardy-Littlewood 
maximal function.
Unfortunately, this simple line of reasoning only works for $d=3$ and in radial symmetry.

\begin{proposition}[Strichartz estimates for $\mb S_0$]
\label{prop:strich}
Let $p\in [2,\infty]$ and $q\in [6,\infty]$ such that $\frac{1}{p}+\frac{3}{q}=\frac12$.
Then we have the bound
\[ \|[\mb S_0(\cdot)\mb f]_1\|_{L^p(\R_+)L^q(\B^3)}\lesssim \|\mb f\|_{\mc H} \]
for all $\mb f\in \mc H$.
As a consequence, we also have
\[ \left \|\int_0^\tau [\mb S_0(\tau-\sigma)\mb h(\sigma,\cdot)]_1 \d\sigma 
\right \|_{L^p_\tau(\R_+)L^q(\B^3)}
\lesssim \|\mb h\|_{L^1(\R_+)\mc H} \]
for all $\mb h\in C([0,\infty),\mc H)\cap L^1(\R_+,\mc H)$.
\end{proposition}

\begin{proof}
Recall d'Alembert's formula which states that 
classical solutions\footnote{By a classical solution we mean $u\in C^2(\R\times [0,\infty))$ 
such that $\partial_r u(t,r)|_{r=0}=0$ for all $t\in \R$.} 
of $(\partial_t^2-\partial_r^2-\frac{2}{r}\partial_r)u(t,r)=0$ satisfy
\begin{align} 
u(t,r)&=\frac{1}{2r}\big [(t+r)u(0, |t+r|)-(t-r)u(0,|t-r|) \big ] \nonumber \\
&\quad +\frac{1}{2r}\int_{|t-r|}^{t+r} s
\partial_0 u(0,s)\d s 
\nonumber \\
&=\frac{1}{2r}\int_{t-r}^{t+r}\partial_s [s u(0,|s|)]\d s+\frac{1}{2r}\int_{|t-r|}^{t+r}
s\partial_0 u(0,s)\d s.
\end{align}
Consequently, we 
obtain for 
\[ \psi(\tau,\rho)=T^\frac12 e^{-\frac12\tau}u(T-Te^{-\tau}, Te^{-\tau}\rho) \]
the formula
\begin{align}
\label{eq:dAlem}
\psi(\tau,\rho)&=\frac{e^{-\frac12 \tau}}{2e^{-\tau}\rho}
\int_{1-e^{-\tau}-e^{-\tau}\rho}^{1-e^{-\tau}+e^{-\tau}\rho}\partial_s [s\psi(0,|s|)]\d s \nonumber \\
&\quad +\frac{e^{-\frac12 \tau}}{2e^{-\tau}\rho}
\int_{|1-e^{-\tau}-e^{-\tau}\rho|}^{1-e^{-\tau}+e^{-\tau}\rho}s[\partial_0 \psi(0,s)
+s\partial_s \psi(0,s)+\tfrac12 \psi(0,s)]\d s
\end{align}
for all $\tau> 0$ and $\rho\in [0,1]$.
Now let $\mb f=(0,f_2)$ with $f_2\in C^1([0,1])$.
In view of the transformations \eqref{eq:utopsi} and \eqref{eq:psitopsi12}, as well as 
Eq.~\eqref{eq:dAlem}, we infer
the explicit representation
\begin{align*} 
[\mb S_0(\tau)\mb f]_1(\rho)=\frac{e^{-\frac12 \tau}}{2e^{-\tau}\rho}
\int_{|1-e^{-\tau}-e^{-\tau}\rho|}^{1-e^{-\tau}+e^{-\tau}\rho} 1_{[0,1]}(s)sf_2(s)\d s
\end{align*}
for all $\tau>0$ and $\rho\in [0,1]$.
Since $1-e^{-\tau}-e^{-\tau}\rho\leq |1-e^{-\tau}-e^{-\tau}\rho|$, we obtain
\begin{align*}
|[\mb S_0(\tau)\mb f]_1(\rho)|&\leq \frac{e^{-\frac12\tau}}{2e^{-\tau}\rho}
\int_{1-e^{-\tau}-e^{-\tau}\rho}^{1-e^{-\tau}+e^{-\tau}\rho}
1_{[0,1]}(s)|sf_2(s)|\d s \\
&\leq e^{-\frac12\tau}\sup_{\rho>0}\left [\frac{1}{2e^{-\tau}\rho}
\int_{1-e^{-\tau}-e^{-\tau}\rho}^{1-e^{-\tau}+e^{-\tau}\rho} 1_{[0,1]}(s)|sf_2(s)|\d s \right ] \\
&=e^{-\frac12\tau}\mc M(1_{[0,1]}|\cdot|f_2)(1-e^{-\tau}),
\end{align*}
where $\mc M$ denotes the Hardy-Littlewood maximal function.
By the $L^2$-boundedness of $\mc M$, see e.g.~\cite{Gra14}, p.~88, Theorem 2.1.6, 
we infer
\begin{align*} 
\int_0^\infty \|[\mb S_0(\tau)\mb f]_1\|_{L^\infty(\B^3)}^2 \d \tau 
&\leq \int_0^\infty \left |\mc M(1_{[0,1]}|\cdot|f_2)(1-e^{-\tau})\right |^2 e^{-\tau} \d \tau \\
&=\int_0^1 \left |\mc M(1_{[0,1]}|\cdot|f_2)(s)\right |^2 \d s \\
&\lesssim \left \|1_{[0,1]}|\cdot|f_2\right \|_{L^2(\R)}^2 \simeq \|f_2\|_{L^2(\B^3)}^2 \\
&\lesssim \|\mb f\|_{\mc H}^2.
\end{align*}
For $\mb f=(f_1,0)$ with $f_1\in C^2([0,1])$ and $f_1'(0)=0$, we obtain
from Eqs.~\eqref{eq:dAlem} and \eqref{eq:utopsi} the representation
\begin{align*} [\mb S_0(\tau)\mb f]_1(\rho)&=\frac{e^{-\frac12\tau}}{2e^{-\tau}\rho}
\int_{1-e^{-\tau}-e^{-\tau}\rho}^{1-e^{-\tau}+e^{-\tau}\rho}\partial_s[s f_1(|s|)]\d s \\
&=\frac{e^{-\frac12\tau}}{2e^{-\tau}\rho}
\int_{1-e^{-\tau}-e^{-\tau}\rho}^{1-e^{-\tau}+e^{-\tau}\rho}1_{[-1,1]}(s)
\partial_s[s f_1(|s|)]\d s .
\end{align*}
Consequently, the same reasoning as above yields the bound
\begin{align*}
 \int_0^\infty \left \|[\mb S_0(\tau)\mb f]_1\right \|_{L^\infty(\B^3)}^2 \d\tau
&\lesssim \left \||\cdot|f_1'(|\cdot|)\right \|_{L^2(-1,1)}^2 
+\left \|f_1(|\cdot|)\right \|_{L^2(-1,1)}^2 \\
&\simeq \left \||\cdot|f_1'\right \|_{L^2(0,1)}^2 
+\left \|f_1\right \|_{L^2(0,1)}^2 \lesssim \|f_1\|_{H^1(\B^3)}^2 \\
&\lesssim \|\mb f\|_{\mc H}^2
\end{align*}
where $\|f_1\|_{L^2(0,1)}\lesssim \|f_1\|_{H^1(\B^3)}$ follows by means of an integration
by parts and the one-dimensional Sobolev embedding $H^1(\frac12,1) 
\hookrightarrow L^\infty(\frac12,1)$.
In summary, we have obtained the Strichartz estimate
\[ \|[\mb S_0(\cdot)\mb f]_1\|_{L^2(\R_+)L^\infty(\B^3)}\lesssim \|\mb f\|_{\mc H} \]
for all $\mb f=(f_1,f_2)\in C^2\times C^1([0,1])$ with $f_1'(0)=0$ and by a
density argument this extends to all $\mb f\in \mc H$.
Furthermore, by the Sobolev embedding $H^1(\B^3)\hookrightarrow L^6(\B^3)$ and the growth
bound from Proposition \ref{prop:gen} we infer
\[ \|[\mb S_0(\tau)\mb f]_1\|_{L^6(\B^3)}\lesssim \|[\mb S_0(\tau)\mb f]_1\|_{H^1(\B^3)}
\lesssim \|\mb S_0(\tau)\mb f\|_{\mc H}\lesssim \|\mb f\|_{\mc H} \]
which implies
$\|[\mb S_0(\cdot)\mb f]_1\|_{L^\infty(\R_+) L^6(\B^3)}\lesssim \|\mb f\|_{\mc H}$.
The general case now follows by interpolation. Indeed, for any $q\in [6,\infty]$ we have
\begin{align*}
 \left \|[\mb S_0(\tau)\mb f]_1\right \|_{L^q(\B^3)}&\leq 
\left \|[\mb S_0(\tau)\mb f]_1\right \|_{L^\infty(\B^3)}^{1-6/q}
\left \|[\mb S_0(\tau)\mb f]_1\right \|_{L^6(\B^3)}^{6/q} \\
&=:\varphi_\infty(\tau)^{1-6/q}\varphi_6(\tau)^{6/q}
\end{align*}
and thus,
\begin{align*}
 \left \|[\mb S_0(\cdot)\mb f]_1\right \|_{L^p(\R_+)L^q(\B^3)}&\leq
 \left \|\varphi_\infty^{1-6/q}\varphi_6^{6/q}\right \|_{L^p(\R_+)} \\
 &\leq \|\varphi_6\|_{L^\infty(\R_+)}^{6/q}\|\varphi_\infty\|_{L^2(\R_+)}^{1-6/q} \\
 &= \left \|[\mb S_0(\cdot)\mb f]_1\right \|_{L^\infty(\R_+)L^6(\B^3)}^{6/q}
  \left \|[\mb S_0(\cdot)\mb f]_1\right \|_{L^2(\R_+)L^\infty(\B^3)}^{1-6/q} \\
  &\lesssim \|\mb f\|_{\mc H},
 \end{align*}
 provided $p(1-\frac{6}{q})=2$, which is equivalent to $\frac{1}{p}+\frac{3}{q}=\frac12$.
 
 For the inhomogeneous estimate we employ Minkowski's inequality which yields
 \begin{align*}
& \left \|\int_0^\tau [\mb S_0(\tau-\sigma)\mb h(\sigma,\cdot)]_1
\d\sigma \right \|_{L^p_\tau(\R_+)L^q(\B^3)} \\
&=\left \|\int_0^\infty 1_{\R_+}(\tau-\sigma)[\mb S_0(\tau-\sigma)\mb h(\sigma,\cdot)]_1
\d\sigma \right \|_{L^p_\tau(\R_+)L^q(\B^3)} \\
&\leq \int_0^\infty \|1_{\R_+}(\tau-\sigma)[\mb S_0(\tau-\sigma)\mb h(\sigma,\cdot)]_1
\|_{L^p_\tau(\R_+)L^q(\B^3)}\d\sigma \\
&\leq \int_0^\infty \|[\mb S_0(\tau)\mb h(\sigma,\cdot)]_1\|_{L^p_\tau(\R_+)L^q(\B^3)}\d\sigma \\
&\lesssim \int_0^\infty \|\mb h(\sigma,\cdot)\|_{\mc H}\d\sigma
 \end{align*}
 by the homogeneous Strichartz estimate.
\end{proof}

\subsection{The linearized evolution}

The operator $\mb L': \mc H\to \mc H$ is bounded and thus, by the Bounded Perturbation
Theorem, $\mb L:=\mb L_0+\mb L'$ generates a semigroup $\mb S(\tau)$.
We also obtain the growth bound
\[ \|\mb S(\tau)\|_{\mc H}\lesssim e^{M\tau} \]
where $M=C\|\mb L'\|_\mc H$ but of course, this is far from optimal.
By a more detailed spectral analysis we obtain the following refined information.

\begin{proposition}
\label{prop:SG}
For the spectrum of $\mb L=\mb L_0+\mb L'$ we have
\[ \sigma(\mb L)\backslash \{1\}\subset \{z\in \C: \Re z\leq 0\} \]
and $1\in \sigma_p(\mb L)$.
The geometric eigenspace of the eigenvalue $1$ is one-dimensional and spanned by 
\[ \mb g(\rho)=\left ( \begin{array}{c}2 \\ 3 \end{array} \right ). \]
Furthermore, there exists a (bounded) projection $\mb P: \mc H\to \langle \mb g\rangle$
such that $[\mb P,\mb S(\tau)]=\mb 0$ for all $\tau\geq 0$, where $\mb S$ is the semigroup
generated by $\mb L$. 
As a consequence, we have $\mb S(\tau)\mb P\mb f=e^{\tau}\mb P\mb f$ for all $\tau\geq 0$
and $\mb f\in \mc H$.
Finally, for any $\epsilon>0$ there exists a constant $C_\epsilon>0$ such that
\[ \|\mb S(\tau)(\mb I-\mb P)\mb f\|_{\mc H}
\leq C_\epsilon e^{\epsilon\tau} \|(\mb I-\mb P)\mb f\|_{\mc H} \]
for all $\tau\geq 0$ and all $\mb f\in \mc H$.
\end{proposition}

\begin{proof}
With $\mb G$ given in Eq.~\eqref{eq:G}, we infer
\[ \mb G \mb L' \mb G^{-1} \mb f(\rho)=\left ( \begin{array}{c}
\frac{15}{4}\int_0^\rho f_2(s)\d s \\ 0 \end{array} \right ) \]
for $\mb f\in L^2(0,1)^2$.
Consequently, $\mb G \mb L'\mb G^{-1}$ is the operator $L'$ from \cite{DonSch12} (with $p=5$)
and the assertions follow from Lemmas 3.5, 3.6, 3.7, and
Proposition 3.9 in \cite{DonSch12}.
\end{proof}

\subsection{Explicit representation of the semigroup}
Let $\mb f \in C^2\times C^1([0,1])$ and for brevity we set 
$\tilde{\mb f}:=(\mb I-\mb P)\mb f$.
Then $\tilde{\mb f}\in C^2\times C^1([0,1])\subset \mc D(\mb L_0)$
since $\rg \mb P=\langle \mb g\rangle$ and $\mb g\in C^\infty\times C^\infty([0,1])$.
By \cite{Kat95}, p.~178, Theorem 6.17, the reduced resolvent 
$\lambda\mapsto \mb R_{\mb L}(\lambda)(\mb I-\mb P)$ has an analytic continuation 
to $\{z\in \C: \Re z>0\}$.
Consequently, by Laplace inversion, we obtain the representation formula
\begin{align*}
 \mb S(\tau)(\mb I-\mb P)\mb f&=\frac{1}{2\pi \I}\lim_{N\to\infty}
\int_{\epsilon-\I N}^{\epsilon+\I N}e^{\lambda\tau}\mb R_{\mb L}(\lambda)(\mb I-\mb P)\mb f\, \d \lambda  \\
&=\frac{1}{2\pi \I}\lim_{N\to\infty}
\int_{\epsilon-\I N}^{\epsilon+\I N}e^{\lambda\tau}\mb R_{\mb L}(\lambda)\tilde{\mb f}\, \d \lambda
\end{align*}
for any $\epsilon>0$, see \cite{EngNag00}, p.~234, Corollary 5.15.
Thus, in order to proceed, we need an explicit expression for the resolvent $\mb R_\mb L(\lambda)$.

If $\mb u=\mb R_\mb L(\lambda)\tilde{\mb f}$ 
then $(\lambda-\mb L)\mb u=\tilde{\mb f}$. Written out, this equation
reads
\[
\left \{
\begin{array}{l}
\rho u_1'(\rho)+(\lambda+\frac{1}{2})u_1(\rho)-u_2(\rho)=\tilde f_1(\rho) \\
-u_1''(\rho)-\frac{2}{\rho}u_1'(\rho)+\rho u_2'(\rho)+(\lambda+\frac{3}{2})
u_2(\rho)-\frac{15}{4}u_1(\rho)=\tilde f_2(\rho)
\end{array} \right . .
\]
Via the first equation we can express $u_2$ in terms of $u_1$ and $\tilde f_1$.
Inserting this into the second equation we find
\begin{align}
\label{eq:specinh}
-(1&-\rho^2)u''(\rho)-\tfrac{2}{\rho}u'(\rho)+3\rho u'(\rho)
+2\lambda \rho u'(\rho)
+\left [\lambda^2+2\lambda+\tfrac{3}{4}\right ]u(\rho)
-\tfrac{15}{4}u(\rho)=F_\lambda(\rho)
\end{align}
with $u=u_1$ and
\[ F_\lambda(\rho)=\rho \tilde f_1'(\rho)+(\lambda+\tfrac{3}{2}) \tilde f_1(\rho)
+\tilde f_2(\rho). \]
Consequently,
\begin{align*} 
[\mb R_{\mb L}(\lambda)\tilde{\mb f}]_1(\rho)=\int_0^1 G(\rho,s;\lambda)
[s\tilde f_1'(s)+(\lambda+\tfrac{3}{2})\tilde f_1(s)+\tilde f_2(s)]\d s 
\end{align*}
where $G$ is the Green function of Eq.~\eqref{eq:specinh}, uniquely defined, as we will see,
by the requirement $u_1\in H^1(\B^3)$.
The first component of $\mb S(\tau)(\mb I-\mb P)\mb f$ is therefore 
given by
\begin{align}
\label{eq:SG}
[\mb S(\tau)(\mb I-\mb P)\mb f]_1(\rho)=\frac{1}{2\pi \I}\lim_{N\to\infty}
\int_{\epsilon-\I N}^{\epsilon+\I N}e^{\lambda\tau}
\int_0^1 G(\rho,s;\lambda)F_\lambda(s)\d s\, \d \lambda 
\end{align}
for any $\epsilon>0$.

\section{Construction of the Green function}
\noindent Our goal in this section is the construction of the Green function
 for Eq.~\eqref{eq:specinh}. In fact, without additional effort,
 most parts of this construction can be carried
 out for the more general equation 
 \begin{align}
\label{eq:specV}
-(1&-\rho^2)u''(\rho)-\tfrac{2}{\rho}u'(\rho)+3\rho u'(\rho)
+2\lambda \rho u'(\rho) 
+\left [\lambda^2+2\lambda+\tfrac{3}{4}\right ]u(\rho)
+V(\rho)u(\rho)=F_\lambda(\rho)
\end{align}
where the ``potential'' $V$ is an arbitrary prescribed function in $C^\infty([0,1])$.
Thus, for future reference, we will consider the more general version Eq.~\eqref{eq:specV}
whenever possible.

\subsection{Preliminaries}
Throughout we will treat the potential $V$ perturbatively. Note that the 
\emph{free equation}, i.e., Eq.~\eqref{eq:specV} with $V=F_\lambda=0$, has the Frobenius indices
$\{0,-1\}$ at $\rho=0$ and $\{0, \frac12-\lambda\}$ at $\rho=1$, respectively.
As usual, it is convenient to remove the first-order derivative by setting
\begin{equation}
\label{eq:utov}
 v(\rho):=\rho(1-\rho^2)^{\frac14+\frac{\lambda}{2}}u(\rho). 
 \end{equation}
Then Eq.~\eqref{eq:specV} with $F_\lambda=0$ is equivalent to
\begin{align}
\label{eq:specv}
 v''(\rho)&+\frac{\lambda(1-\lambda)+\frac34}{(1-\rho^2)^2}
v(\rho) =\frac{V(\rho)}{1-\rho^2}v(\rho).
\end{align}
Note that if $v(\cdot;\lambda)$ is a solution to Eq.~\eqref{eq:specv} then so is
$v(\cdot;1-\lambda)$.
This convenient symmetry will simplify many computations in the sequel.

\subsection{Construction of a fundamental system}
Eq.~\eqref{eq:specv} with $V=0$ has an explicit 
fundamental system given by
\begin{align*}
 \psi_1(\rho;\lambda)&=(1+\rho)^{\frac34-\frac{\lambda}{2}}(1-\rho)^{\frac14+\frac{\lambda}{2}} \\
 \tilde \psi_1(\rho;\lambda)&=\psi_1(\rho;1-\lambda)=
 (1+\rho)^{\frac14+\frac{\lambda}{2}}(1-\rho)^{\frac34-\frac{\lambda}{2}}.
\end{align*}
Strictly speaking, this is a fundamental system only if $\lambda\not=\frac12$.
However, we are interested in $\lambda$ close to the imaginary axis, so
this issue does not bother us.
We define a third solution $\psi_0$ by
\[ \psi_0(\rho;\lambda):=\psi_1(\rho;\lambda)-\tilde \psi_1(\rho;\lambda) \]
and note that $\psi_0(0;\lambda)=0$.
We now show that the fundamental system $\{\psi_1,\tilde\psi_1\}$ can be perturbed to
yield a fundamental system for Eq.~\eqref{eq:specv}.
In the following we construct solutions to ODEs by Volterra iterations.
For the basic existence theory in this context we refer to \cite{DeiTru79} or
\cite{SchSofSta10}.
Furthermore, for brevity it is useful to introduce the following notation.

\begin{definition}
\label{def:O}
For a function $f: I\subset \R\to \C$, $x_0\in \overline I$, and $\alpha\in \R$, we write
$f(x)=\O((x-x_0)^\alpha)$ if
\[ |\partial_x^j f(x)|\leq C_j |x-x_0|^{\alpha-j} \]
for all $x\in I$ and $j\in \N_0$.
Similarly, $f(x)=\O(\langle x\rangle^\alpha)$ means
\[ |\partial_x^j f(x)|\leq C_j \langle x\rangle^{\alpha-j} \]
for all $x\in I$ and $j\in \N_0$.
If $f: \R\to \C$ is odd, we indicate this by using the symbol $\Oo$, e.g.~$f(x)=\Oo(\langle x \rangle^{-1})$.

An analogous notation is used for functions of more than one variable, e.g.~for
$f: U\subset \R^2\to \C$ and $\alpha,\beta \in \R$, we write $f(x,y)=\O(x^\alpha \langle y\rangle^\beta)$
if
\[ |\partial_x^j \partial_y^k f(x,y)|\leq C_{j,k}|x|^{\alpha-j}\langle y\rangle^{\beta-k} \]
for all $(x,y)\in U$ and $j,k\in \N_0$.
Functions of this type are said to \emph{behave like symbols} or to be of \emph{symbol type}.
\end{definition}

\begin{remark}
As a consequence of the Leibniz rule, symbol behavior is stable under algebraic operations, 
e.g.~$\O(x^\alpha)\O(x^\beta)=\O(x^{\alpha+\beta})$.
In addition, if $f(x)=\O(x^\alpha)$, the chain rule implies $f(\langle x\rangle^{-1})=\O(\langle x\rangle^{-\alpha})$.
\end{remark}

\begin{lemma}
\label{lem:v1}
With $\lambda=\epsilon+\I\omega$, Eq.~\eqref{eq:specv} has a solution
$v_1(\cdot;\lambda)$ of the
form\footnote{Of course, the error function $\O((1-\rho)\langle\omega\rangle^{-2})$ depends
on $\epsilon$ as well but since this dependence is inessential, we suppress it
in the notation.}
\begin{align*} 
v_1(\rho;\lambda)&=\psi_1(\rho;\lambda)[1+\tfrac{1}{1-2\lambda}a_1(\rho)+\O((1-\rho)\langle\omega \rangle^{-2})] 
\end{align*}
for all $\rho \in [0,1)$, $\omega \in \R$, and $\epsilon\in [0,\frac13]\cup [\frac23,1]$, where
\[ a_1(\rho)=-\int_\rho^1 V(s)\d s. \]
Furthermore, another solution $\tilde v_1(\cdot;\lambda)$ is given by
\[ \tilde v_1(\rho;\lambda)=\tilde \psi_1(\rho;\lambda)
[1-\tfrac{1}{1-2\lambda}a_1(\rho)+\O((1-\rho)\langle\omega \rangle^{-2})]. \]
\end{lemma}

\begin{proof}
We set
\[ W(\lambda):=W(\psi_1(\cdot;\lambda), \tilde \psi_1(\cdot;\lambda))=-1+2\lambda. \]
The variation of constants formula suggests to look for a solution $v_1$ of the integral
equation
\begin{align*} v_1(\rho;\lambda)=\psi_1(\rho;\lambda)&+\frac{\psi_1(\rho;\lambda)}{W(\lambda)}
\int_\rho^{\rho_1}\tilde \psi_1(s;\lambda)\frac{V(s)}{1-s^2}v_1(s;\lambda)\d s \\
&-\frac{\tilde \psi_1(\rho;\lambda)}{W(\lambda)}\int_\rho^{\rho_1}
\psi_1(s;\lambda)\frac{V(s)}{1-s^2}v_1(s;\lambda)\d s
\end{align*}
where $\rho_1\in [0,1]$ is a constant that will be chosen later.
Since $|\psi_1(\rho;\lambda)|>0$ for all $\rho \in (0,1)$ and $\lambda\in \C$, 
we may set $h_1:=\frac{v_1}{\psi_1}$
which yields the Volterra equation
\begin{equation}
\label{eq:h1} h_1(\rho;\lambda)=1+\int_\rho^{\rho_1}K(\rho,s;\lambda)h_1(s;\lambda)\d s 
\end{equation}
with the kernel
\[ K(\rho,s;\lambda)=\frac{1}{W(\lambda)}\left [\psi_1(s;\lambda)\tilde \psi_1(s;\lambda)
-\frac{\tilde\psi_1(\rho;\lambda)}{\psi_1(\rho;\lambda)}\psi_1(s;\lambda)^2 \right ]
\frac{V(s)}{1-s^2}. \]
Explicitly, we have
\begin{align*}
\psi_1(s;\lambda)\tilde\psi_1(s;\lambda)&=1-s^2 \\
\frac{\tilde \psi_1(\rho;\lambda)}{\psi_1(\rho;\lambda)}\psi_1(s;\lambda)^2
&=\left (\frac{1-\rho}{1+\rho} \right )^{\frac12-\lambda}
(1+s)^2 \left (\frac{1-s}{1+s} \right )^{\frac12+\lambda}
\end{align*}
which implies
\[ \left |\frac{\tilde\psi_1(\rho;\lambda)}{\psi_1(\rho;\lambda)}\psi_1(s;\lambda)^2\right |
\lesssim (1-\rho)^{\frac12-\Re\lambda}(1-s)^{\frac12+\Re\lambda}
\lesssim (1-\rho)^\frac12 (1-s)^\frac12 \]
for all $0\leq \rho\leq s\leq 1$ and $\Re\lambda\geq 0$.
Consequently, we obtain the bound
\[ |K(\rho,s;\lambda)|\lesssim (1-\rho)^\frac12 (1-s)^{-\frac12}\langle \lambda\rangle^{-1} \]
for all $0\leq \rho\leq s< 1$.
This allows us to choose $\rho_1=1$ and we infer
\[ \int_0^1 \sup_{\rho\in (0,1)}|K(\rho,s;\lambda)|\d s\lesssim \langle \lambda\rangle^{-1}.  \]
Thus, Eq.~\eqref{eq:h1} has a solution $h_1$ satisfying
\[ |h_1(\rho;\lambda)-1|\lesssim \int_\rho^1 |K(\rho,s;\lambda)|\d s \lesssim (1-\rho)
\langle \lambda \rangle^{-1} \]
for all $\rho\in [0,1]$.
Re-inserting this into Eq.~\eqref{eq:h1} we find
\begin{align*} 
h_1(\rho;\lambda)&=1+\int_\rho^1 K(\rho,s;\lambda)[1+O((1-s)\langle\lambda\rangle^{-1})]
\d s \\
&=1+\int_\rho^1 K(\rho,s;\lambda)\d s+O((1-\rho)^2 \langle\lambda\rangle^{-2}).
\end{align*}
Note that
\begin{align*}
\int_\rho^1 K(\rho,s;\lambda)\d s&=\frac{1}{W(\lambda)}\left [\int_\rho^1 V(s)\d s
-\left (\frac{1-\rho}{1+\rho}\right )^{\frac12-\lambda}
I(\rho;\lambda) \right ]
\end{align*}
where
\[ I(\rho;\lambda)=\int_\rho^1 \left (\frac{1-s}{1+s} \right )^{-\frac12+\lambda}
V(s)\d s. \]
We have
\[ \left (\frac{1-s}{1+s}\right )^{-\frac12+\lambda}=-\frac{(1+s)^2}{1+2\lambda}
\partial_s \left (\frac{1-s}{1+s} \right )^{\frac12+\lambda} \]
and thus,
\begin{align*}
I(\rho;\lambda)&=-\frac{1}{1+2\lambda}\int_\rho^1 
\partial_s \left (\frac{1-s}{1+s} \right )^{\frac12+\lambda}
 (1+s)^2 V(s)\d s \\
&=\frac{1}{1+2\lambda}\left (\frac{1-\rho}{1+\rho} \right )^{\frac12+\lambda}
 (1+\rho)^2 V(\rho)
+\frac{1}{1+2\lambda}\int_\rho^1 \left (\frac{1-s}{1+s}\right )^{\frac12+\lambda}
\partial_s[(1+s)^2V(s)]\d s.
\end{align*}
Consequently, we infer
\begin{align*}
 W(\lambda)\int_\rho^1 K(\rho,s;\lambda)\d s&=\int_\rho^1 V(s)\d s
-\frac{1}{1+2\lambda}(1-\rho^2) V(\rho) \\
&\quad -\frac{1}{1+2\lambda}\left (\frac{1-\rho}{1+\rho}\right )^{\frac12-\lambda}
\int_\rho^1 \left (\frac{1-s}{1+s}\right )^{\frac12+\lambda}
\partial_s[(1+s)^2V(s)]\d s \\
&=-a_1(\rho)+O((1-\rho)\langle\lambda\rangle^{-1}).
\end{align*}

The bounds on the derivatives 
are proved as follows.
One sets $\varphi(s):=-\frac12 \log(1-s)+\frac12 \log(1+s)$ and 
notes that $\varphi(s)>0$, $\varphi'(s)=\frac{1}{1-s^2}\geq 1$ for all $s\in (0,1)$.
Next, one observes that the inverse of the 
diffeomorphism $\varphi: (0,1)\to (0,\infty)$, given explicitly by
\[ \varphi^{-1}(y)=\frac{e^{2y}-1}{e^{2y}+1}, \] satisfies the bounds
\[ |\partial_y^j \varphi^{-1}(y)|\leq C_j e^{-2y} \]
for all $y\geq 0$ and $j\in \N$.
Then one introduces $y=\varphi(s)$ as a new integration variable and rescales to remove
the $\omega$-dependence from the oscillatory factor.
As a consequence, no derivatives fall on oscillatory terms. 
A simple induction then yields the stated bounds.
To be more precise, recall that the oscillatory
part of $\int_\rho^1 K(\rho,s;\lambda)\d s$ 
consists of the term
\begin{align*}
&\frac{1-\rho}{1+\rho}\int_\rho^1 \left (\frac{1-\rho}{1+\rho}\right )^{-\frac12-\lambda}
\left (\frac{1-s}{1+s}\right )^{\frac12+\lambda}V_1(s)\d s 
 =\frac{1-\rho}{1+\rho}
\int_\rho^1 e^{-(1+2\lambda)[\varphi(s)-\varphi(\rho)]}V_1(s)\d s
\end{align*}  
where $V_1(s):=\partial_s [(1+s)^2 V(s)]$.
Thus, we consider
\begin{align*} J_\epsilon(\rho;\omega):&=\int_\rho^1 e^{(-1-2\epsilon-2\I\omega)
[\varphi(s)-\varphi(\rho)]} V_1(s)\d s \\
&=\int_{\varphi(\rho)}^\infty e^{(-1-2\epsilon-2\I\omega)[y-\varphi(\rho)]}
V_1(\varphi^{-1}(y))(\varphi^{-1})'(y)\d y \\
&=\int_0^\infty e^{(-1-2\epsilon-2\I\omega)y}\tilde V_1(y+\varphi(\rho))\d y
\end{align*}
where $\tilde V_1(y):=V(\varphi^{-1}(y))(\varphi^{-1})'(y)$.
Note that in this representation it is already evident that no $\rho$-derivatives fall on
oscillatory factors.
Furthermore, 
\[ \partial_\rho^j\partial_\omega^k J_\epsilon(\rho;\omega) 
=(-2\I)^k\int_0^\infty y^k e^{(-1-2\epsilon-2\I\omega)y}\partial_\rho^j \tilde V_1(y+\varphi(\rho))\d y \]
which yields $|\partial_\rho^j \partial_\omega^k J_\epsilon(\rho;\omega)|\leq C_{j,k}(1-\rho)^{-j}$.
Thus, we may safely assume $|\omega|\geq 1$.
Then we rescale to obtain
\[
J_\epsilon(\rho;\omega)=\omega^{-1}\int_0^\infty e^{(-1-2\epsilon)y/\omega}e^{-2\I y}
\tilde V_1(\tfrac{y}{\omega}+\varphi(\rho))\d y
\]
which yields
\[ | \partial_\rho^j\partial_\omega^k J_\epsilon(\rho;\omega)|\leq C_{j,k} (1-\rho)^{-j}\omega^{-k}. \]
Consequently, the claimed bounds follow inductively.
The second solution is given by 
$\tilde v_1(\cdot;\lambda)=v_1(\cdot;1-\lambda)$.
\end{proof}

Next, we perturb $\psi_0$ to obtain a third solution 
of Eq.~\eqref{eq:specv}.

\begin{lemma}
\label{lem:v0}
There exists a $\delta_0>0$ such that
Eq.~\eqref{eq:specv}, with $\lambda=\epsilon+\I\omega$, has a solution $v_0$ of the
form
\[ 
v_0(\rho;\lambda)=\psi_0(\rho;\lambda)[1+\O(\rho^2\langle\omega\rangle^0)] 
\]
for all $\rho \in [0,\delta_0\langle\omega\rangle^{-1}]$, $\omega\in \R$, and
$\epsilon\in [0,\frac13]$.
\end{lemma}

\begin{proof}
We use the fundamental system $\{\psi_0,\psi_1\}$ and set
\[ W(\lambda):=W(\psi_0(\cdot;\lambda),\psi_1(\cdot;\lambda))
=-W(\tilde \psi_1(\cdot;\lambda),\psi_1(\cdot;\lambda))=-1+2\lambda. \]
Motivated by the variation of constants formula, we consider the integral equation
\begin{align*}
 v_0(\rho;\lambda)=\psi_0(\rho;\lambda)&-\frac{\psi_0(\rho;\lambda)}{W(\lambda)}
 \int_0^\rho \psi_1(s;\lambda)\frac{V(s)}{1-s^2}v_0(s;\lambda)\d s \\
 &+\frac{\psi_1(\rho;\lambda)}{W(\lambda)}
 \int_0^\rho \psi_0(s;\lambda)\frac{V(s)}{1-s^2}v_0(s;\lambda)\d s.
\end{align*}
Since $\Re\lambda\leq \frac13$ and $\frac{1-\rho}{1+\rho}\leq 1$, we have
\begin{align*} |\psi_1(\rho;\lambda)|&=\left |\left (\frac{1-\rho}
{1+\rho}\right )^{\frac14+\frac{\lambda}{2}}(1+\rho)\right |=
 \left (\frac{1-\rho}
{1+\rho}\right )^{\frac14+\frac{\Re\lambda}{2}}(1+\rho)
\geq \left (\frac{1-\rho}
{1+\rho}\right )^{\frac{5}{12}}(1+\rho) \\
|\tilde \psi_1(\rho;\lambda)|&=\left |\left (\frac{1-\rho}
{1+\rho}\right )^{\frac34-\frac{\lambda}{2}}(1+\rho)\right |=
 \left (\frac{1-\rho}
{1+\rho}\right )^{\frac34-\frac{\Re\lambda}{2}}(1+\rho)
\leq \left (\frac{1-\rho}
{1+\rho}\right )^{\frac{7}{12}}(1+\rho)
\end{align*}
and thus,
\begin{align*}
 |\psi_0(\rho;\lambda)|&\geq |\psi_1(\rho;\lambda)|-|\tilde \psi_1(\rho;\lambda)| 
\geq (1+\rho)\left (\frac{1-\rho}{1+\rho}\right )^\frac{5}{12}\left [
 1-\left (\frac{1-\rho}{1+\rho} \right )^{\frac16}\right ]>0
 \end{align*}
for all $\rho\in (0,1)$.
Consequently, we may set
$h_0:=\frac{v_0}{\psi_0}$ which leads to the Volterra equation
\begin{equation}
\label{eq:h0}
 h_0(\rho;\lambda)=1+\int_0^\rho K(\rho,s;\lambda)h_0(s;\lambda)\d s 
 \end{equation}
with the kernel
\[ K(\rho,s;\lambda)=\frac{1}{W(\lambda)}\left [
\frac{\psi_1(\rho;\lambda)}{\psi_0(\rho;\lambda)}\psi_0(s;\lambda)^2
-\psi_0(s;\lambda)\psi_1(s;\lambda) \right ]\frac{V(s)}{1-s^2}.
\]
By Taylor expansion we find 
$\psi_0(\rho;\lambda)=(1-2\lambda)\rho[1+O(\rho\langle\lambda\rangle)]$ and thus,
$\psi_0(\rho;\lambda)^{-1}=(1-2\lambda)^{-1}\rho^{-1}[1+O(\rho\langle\lambda\rangle)]$
for all $\rho\in [0,\delta_0\langle\omega\rangle^{-1}]$ and $\lambda=\epsilon+\I\omega$,
provided $\delta_0>0$ is chosen sufficiently small.
Consequently, we obtain $|K(\rho,s;\lambda)|\lesssim s$ 
for all $0\leq s\leq \rho\leq \delta_0\langle\omega\rangle^{-1}$ 
and $\lambda=\epsilon+\I\omega$.
This yields
\[ \int_0^{\delta_0\langle\omega\rangle^{-1}}\sup_{\rho\in (s,\delta_0\langle
\omega\rangle^{-1})}|K(\rho,s;\lambda)|\d s\lesssim 
\langle\omega\rangle^{-2}. \]
Hence, we obtain the existence of a solution $h_0$ to Eq.~\eqref{eq:h0} with the
 bound
 \[ \left |h_0(\rho;\lambda)-1\right |\lesssim \rho^2 \]
 and all functions behave like symbols under differentiation with respect to
 $\rho$ and $\omega$. 
\end{proof}

In order to gain control over the solution $v_0$ near the endpoint $\rho=1$, we express
$v_0$ in terms of the fundamental system $\{v_1,\tilde v_1\}$.
Recall that $\Oo$ denotes an odd function of symbol type.

\begin{lemma}
\label{lem:v0v1}
The solution $v_0$ from Lemma \ref{lem:v0} has the representation
\[ v_0(\rho;\lambda)=[1+\Oo(\langle\omega\rangle^{-1})+\O(\langle\omega\rangle^{-2})]v_1(\rho;\lambda)
-[1+\Oo(\langle\omega\rangle^{-1})+\O(\langle\omega\rangle^{-2})]\tilde v_1(\rho;\lambda) \]
for all $\rho\in [0,1]$, $\epsilon\in [0,\frac13]$, $\omega\in \R$, and $\lambda=\epsilon+\I\omega$.
\end{lemma}

\begin{proof}
Evaluation at $\rho=1$ yields
\begin{align*}
 W(v_1(\cdot;\lambda),\tilde v_1(\cdot;\lambda))&=
W(\psi_1(\cdot;\lambda),\tilde \psi_1(\cdot;\lambda))[1+\O((1-\rho)\langle\omega\rangle^{-1})] \\
&\quad +\psi_1(\rho;\lambda)\tilde \psi_1(\rho;\lambda)\O((1-\rho)^0\langle \omega\rangle^{-1}) \\
&=W(\psi_1(\cdot;\lambda),\tilde \psi_1(\cdot;\lambda)) \\
&=-1+2\lambda
 \end{align*}
 and thus, there exist $a(\lambda)$ and $b(\lambda)$ such that
 \[ v_0(\rho;\lambda)=a(\lambda)v_1(\rho;\lambda)+b(\lambda)\tilde v_1(\rho;\lambda). \]
The connection coefficients are given by 
 \begin{align*}
 a(\lambda)&=\frac{W(v_0(\cdot;\lambda),\tilde v_1(\cdot;\lambda))}
 {W(v_1(\cdot;\lambda),\tilde v_1(\cdot;\lambda))} \\
 b(\lambda)&=\frac{W(v_0(\cdot;\lambda),v_1(\cdot;\lambda))}
 {W(\tilde v_1(\cdot;\lambda),v_1(\cdot;\lambda))}
 \end{align*}
 and by evaluation at $\rho=\delta_0\langle\omega\rangle^{-1}$, where
 $\lambda=\epsilon+\I\omega$ and $\delta_0>0$ is from Lemma \ref{lem:v0}, we find
 \begin{align*}
 W(v_0(\cdot;\lambda),\tilde v_1(\cdot;\lambda))&=
 W(\psi_0(\cdot;\lambda),\tilde \psi_1(\cdot;\lambda)) \\
 &\quad \times [1+\O(\langle\omega\rangle^{-2})]
[1-\tfrac{1}{1-2\lambda}a_1(\delta_0\langle\omega\rangle^{-1})
+\O(\langle\omega\rangle^{-2})] \\
 &\quad +\psi_0(\delta_0\langle\omega\rangle^{-1};\lambda)
 \tilde\psi_1(\delta_0\langle\omega\rangle^{-1};\lambda)\O(\langle\omega\rangle^{-1}) \\
 &=(-1+2\lambda)[1+\Oo(\langle\omega\rangle^{-1})+\O(\langle\omega\rangle^{-2})]
 \end{align*}
 since $a_1(\delta_0 \langle\omega\rangle^{-1})=\O(\langle\omega\rangle^0)$ is an even function
 of $\omega$ and
 \[ \frac{1}{1-2\lambda}=\frac{1-2\epsilon+2\I\omega}{(1-2\epsilon)^2+4\omega^2}=\Oo(\langle\omega\rangle^{-1})
 +\O(\langle\omega\rangle^{-2}). \]
Analogously,
 \[ W(v_0(\cdot;\lambda),v_1(\cdot;\lambda))=(-1+2\lambda)
 [1+\Oo(\langle\omega\rangle^{-1})+\O(\langle\omega\rangle^{-2})]. \]
\end{proof}

\subsection{The Green function}
Now we return to the case $V(\rho)=-\frac{15}{4}$.
In view of the transformation \eqref{eq:utov} we set
\begin{align*}
u_0(\rho;\lambda)&:=\rho^{-1}(1-\rho^2)^{-\frac14-\frac{\lambda}{2}}
v_0(\rho;\lambda) \\
u_1(\rho;\lambda)&:=\rho^{-1}(1-\rho^2)^{-\frac14-\frac{\lambda}{2}}
v_1(\rho;\lambda).
\end{align*}
By construction, for $j\in \{0,1\}$,
$u_j(\cdot;\lambda)$ is a solution of 
Eq.~\eqref{eq:specinh} with $F_\lambda=0$.
Note that by standard ODE regularity theory we have $u_j(\cdot;\lambda)\in C^\infty(0,1)$.
Near $\rho=0$ we have the asymptotic behavior
\begin{align*} 
\lim_{\rho\to 0+}u_0(\rho;\lambda)&=\lim_{\rho\to 0+} \frac{v_0(\rho;\lambda)}{\rho}
=\lim_{\rho\to0+}\frac{\psi_0(\rho;\lambda)}{\rho}=1-2\lambda 
\end{align*}
and thus, $u_0(\cdot;\lambda)\in C[0,1) \cap C^\infty(0,1)$.
Furthermore, we note the bound 
$|u_0'(\rho;\lambda)|\lesssim \rho^{-1}\langle\lambda\rangle$ for $\rho\in [0,\delta_0\langle\omega\rangle^{-1}]$,
$\lambda=\epsilon+\I\omega$ and $\delta_0>0$ sufficiently small (see Lemma \ref{lem:v0}).
Thus, we also have $u_0(\cdot;\lambda)\in H^1(\B^3_{1-\delta})$ for any $\delta\in (0,1)$.
Near $\rho=1$, on the other hand, we have
\begin{align*} \lim_{\rho\to 1-}u_1(\rho;\lambda)&=\lim_{\rho\to 1-}[(1-\rho^2)^{-\frac14
-\frac{\lambda}{2}}v_1(\rho;\lambda)] \\
&=\lim_{\rho\to 1-}
[(1-\rho^2)^{-\frac14
-\frac{\lambda}{2}}\psi_1(\rho;\lambda)] \\
&=2^{\frac12-\lambda} \\
\lim_{\rho\to 1-}u_1'(\rho;\lambda)&=O(\langle \lambda\rangle).
\end{align*}
This shows $u_1(\cdot;\lambda)\in C^1(0,1]\cap C^\infty(0,1)$ and in particular, we
infer $u_1(\cdot;\lambda)\in H^1(\B^3\backslash \{0\})$.
For the construction of the Green function we also need to know the Wronskian
of $u_0$ and $u_1$.

\begin{lemma}
\label{lem:W}
We have
\[ W(u_0(\cdot;\lambda),u_1(\cdot;\lambda))(\rho)=
(-1+2\lambda)w_0(\lambda) \rho^{-2}(1-\rho^2)^{-\frac12-\lambda} \]
with
\[
w_0(\epsilon+\I\omega)=1+\Oo(\langle\omega\rangle^{-1})+\O(\langle\omega\rangle^{-2}) \]
for all $\rho\in (0,1)$, $\epsilon\in [0,\frac13]$, and $\omega\in \R$.
Furthermore, 
$|w_0(\lambda)|\gtrsim 1$ for all $\lambda \in \C$ with $\Re\lambda\in [0,\frac13]$.
\end{lemma}

\begin{proof}
By definition, we have
\[ W(u_0(\cdot;\lambda),u_1(\cdot;\lambda))(\rho)=
\rho^{-2}(1-\rho^2)^{-\frac12-\lambda}W(v_0(\cdot;\lambda),v_1(\cdot;\lambda)) \]
and from the proof of Lemma \ref{lem:v0v1} we recall
\begin{equation}
\label{eq:Wlg} W(v_0(\cdot;\lambda),v_1(\cdot;\lambda))=(-1+2\lambda)
[1+\Oo(\langle\omega\rangle^{-1})+\O(\langle\omega\rangle^{-2})].
\end{equation}
Consequently, we obtain the stated form of the Wronskian.
Upon setting 
\[ v(\rho)=(1-\rho^2)^{\frac14+\frac{\lambda}{2}}w(\rho^2), \] 
we see that Eq.~\eqref{eq:specv} with $V(\rho)=-\frac{15}{4}$ is equivalent to
the hypergeometric differential equation
\begin{equation}
\label{eq:hypgeo}
 z(1-z)w''(z)+[c-(a+b+1)z]w'(z)-ab w(z)=0 
 \end{equation}
where $z=\rho^2$, $a=\frac{\lambda}{2}-1$, $b=\frac{\lambda}{2}+1$, and $c=\frac12$.
Eq.~\eqref{eq:hypgeo} has a fundamental system 
given by
\begin{align*}
 w_0(z;\lambda)&=z^\frac12 {}_2F_1(a-c+1,b-c+1,2-c;z) \\
 \tilde w_0(z;\lambda)&={}_2F_1(a,b,c;z),
\end{align*}
where ${}_2F_1$ denotes the standard hypergeometric function, see \cite{DLMF}. Another fundamental system is given by
\begin{align*}
w_1(z;\lambda)&={}_2F_1(a,b,a+b+1-c; 1-z) \\
\tilde w_1(z;\lambda)&=(1-z)^{c-a-b}{}_2F_1(c-a,c-b,c-a-b+1; 1-z).
\end{align*}
From $\lim_{\rho\to 0+}v_0(\rho;\lambda)=0$ and the asymptotic form
of $v_1(\rho;\lambda)$ as $\rho\to 1-$, we see that
\[ v_j(\rho;\lambda)=(1-\rho^2)^{\frac14+\frac{\lambda}{2}}w_j(\rho^2;\lambda),\qquad j\in \{0,1\}. \]
Consequently, $W(v_0(\cdot;\lambda),v_1(\cdot;\lambda))=0$ if and only if
$W(w_0(\cdot;\lambda),w_1(\cdot;\lambda))=0$.
The connection formula \cite{DLMF}
\[ w_1(z;\lambda)=\frac{\Gamma(1-c)\Gamma(a+b-c+1)}{\Gamma(a-c+1)\Gamma(b-c+1)}\tilde w_0(z;\lambda)
+\frac{\Gamma(c-1)\Gamma(a+b-c+1)}{\Gamma(a)\Gamma(b)}w_0(z;\lambda) \]
shows that $W(w_0(\cdot;\lambda),w_1(\cdot;\lambda))=0$ if and only if
\begin{align*} -a+c-1&=-\tfrac{\lambda}{2}+\tfrac12 \in \N_0 \\
& \mbox{ or} \\
-b+c-1&=-\tfrac{\lambda}{2}-\tfrac32
 \in \N_0. 
 \end{align*}
 Neither of these conditions is satisfied if $\Re\lambda\in [0,\frac13]$ and thus,
 in view of Eq.~\eqref{eq:Wlg}, we infer
 \[ |w_0(\lambda)|=\left |\frac{W(v_0(\cdot;\lambda),v_1(\cdot;\lambda))}{1-2\lambda}\right |
 \gtrsim 1 \]
 for all $\lambda$ with $\Re\lambda\in [0,\frac13]$.
\end{proof}

\begin{corollary}
\label{cor:W}
For the function $w_0$ from Lemma \ref{lem:W} we have the representation
\[ \tfrac{1}{w_0(\lambda)}=1+\Oo(\langle\omega\rangle^{-1})+\O(\langle\omega\rangle^{-2}) \]
for all $\epsilon\in [0,\frac13]$, $\omega\in \R$, and $\lambda=\epsilon+\I\omega$.
\end{corollary}

\begin{proof}
With $\lambda=\epsilon+\I\omega$ we have
\begin{align*} \frac{1}{w_0(\lambda)}&=\frac{\overline{w_0(\lambda)}}{|w_0(\lambda)|^2}
=[1+\Oo(\langle\omega\rangle^{-1})+\O(\langle\omega\rangle^{-2})][1+\O(\langle\omega\rangle^{-1})] \\
&=1+\Oo(\langle\omega\rangle^{-1})+\O(\langle\omega\rangle^{-2}).
\end{align*}
\end{proof}

Lemma \ref{lem:W} implies that
$\{u_0(\cdot;\lambda),u_1(\cdot;\lambda)\}$ is a fundamental
system for the homogeneous version of Eq.~\eqref{eq:specinh}.
Thus, by the variation of constants formula, the general solution of Eq.~\eqref{eq:specinh}
is given by
\begin{align*} u(\rho;\lambda)=&c_0(\lambda)u_0(\rho;\lambda)+c_1(\lambda)u_1(\rho;\lambda) \\
&-\int_0^\rho \frac{u_0(s;\lambda)u_1(\rho;\lambda)}{W(u_0(\cdot;\lambda),
u_1(\cdot;\lambda))(s)}\frac{F_\lambda(s)}{1-s^2}\d s \\
&-\int_\rho^1 \frac{u_0(\rho;\lambda)u_1(s;\lambda)}{W(u_0(\cdot;\lambda),
u_1(\cdot;\lambda))(s)}\frac{F_\lambda(s)}{1-s^2}\d s 
\end{align*}
where $c_j(\lambda)\in \C$, $j\in \{0,1\}$, can be chosen freely.
Since $u_0(\cdot;\lambda)\notin H^1(\B^3)$ and $u_1(\cdot;\lambda)\notin H^1(\B^3)$,
there is a unique solution 
\begin{align*} u(\rho;\lambda)=&-\int_0^\rho \frac{u_0(s;\lambda)u_1(\rho;\lambda)}{W(u_0(\cdot;\lambda),
u_1(\cdot;\lambda))(s)}\frac{F_\lambda(s)}{1-s^2}\d s \\
&-\int_\rho^1 \frac{u_0(\rho;\lambda)u_1(s;\lambda)}{W(u_0(\cdot;\lambda),
u_1(\cdot;\lambda))(s)}\frac{F_\lambda(s)}{1-s^2}\d s 
\end{align*}
which belongs to $H^1(\B^3)$. 
As a consequence, the Green function for Eq.~\eqref{eq:specinh} is given by
\begin{equation} 
\label{eq:Green}
G(\rho,s;\lambda):=\frac{s^2(1-s^2)^{-\frac12+\lambda}}{(1-2\lambda)w_0(\lambda)}
\left \{ 
\begin{array}{l}
u_0(\rho;\lambda)u_1(s;\lambda)\mbox{ if }\rho\leq s \\
u_1(\rho;\lambda)u_0(s;\lambda)\mbox{ if }\rho \geq s
\end{array} \right . 
\end{equation}
and with this $G$, the representation formula \eqref{eq:SG} 
holds.

\subsection{Decomposition}
We set
\begin{align}
\label{def:phi1}
\phi_1(\rho;\lambda)&:=\rho^{-1}(1-\rho^2)^{-\frac14-\frac{\lambda}{2}}
\psi_1(\rho;\lambda)=\rho^{-1}(1+\rho)^{\frac12-\lambda} \nonumber \\
\tilde\phi_1(\rho;\lambda)&:=\rho^{-1}(1-\rho^2)^{-\frac14-\frac{\lambda}{2}}
\tilde\psi_1(\rho;\lambda)=\rho^{-1}(1-\rho)^{\frac12-\lambda}
\end{align}
as well as
\begin{align} 
\label{def:phi0}
\phi_0(\rho;\lambda):=\phi_1(\rho;\lambda)-\tilde\phi_1(\rho;\lambda)=
\rho^{-1}\left [ (1+\rho)^{\frac12-\lambda}-(1-\rho)^{\frac12-\lambda} \right ] .
\end{align}
In view of the transformation \eqref{eq:utov}, 
$\{\phi_0(\cdot;\lambda),\phi_1(\cdot;\lambda)\}$ is a fundamental system for
Eq.~\eqref{eq:specV} with $V=F_\lambda=0$ and
\begin{equation*} 
G_0(\rho,s;\lambda):=\frac{s^2(1-s^2)^{-\frac12+\lambda}}{1-2\lambda}
\left \{ 
\begin{array}{l}
\phi_0(\rho;\lambda)\phi_1(s;\lambda)\mbox{ if }\rho\leq s \\
\phi_1(\rho;\lambda)\phi_0(s;\lambda)\mbox{ if }\rho \geq s
\end{array} \right . 
\end{equation*}
is the corresponding Green function for the free wave equation.
In the following we use a smooth cut-off $\chi: \R\to [0,1]$ that satisfies
$\chi(x)=1$ for $|x|\leq \frac{\delta_0}{2}$ and $\chi(x)=0$ for $|x|\geq \delta_0$,
where $\delta_0>0$ is from Lemma \ref{lem:v0}.

\begin{lemma}
\label{lem:decomp}
We have the decomposition
\[ G(\rho,s;\lambda)=G_0(\rho,s;\lambda)+\sum_{n=1}^6 G_n(\rho,s;\lambda) \]
where 
\begin{align*}
G_1(\rho,s;\lambda)&=1_{\R_+}(s-\rho)\chi(\rho\langle\omega\rangle)s^2(1-s^2)^{-\frac12+\lambda}
\frac{\phi_0(\rho;\lambda)\phi_1(s;\lambda)\gamma_1(\rho,s;\lambda)}{1-2\lambda} \\
G_2(\rho,s;\lambda)&=1_{\R_+}(s-\rho)[1-\chi(\rho\langle\omega\rangle)]s^2(1-s^2)^{-\frac12+\lambda}
\frac{\phi_1(\rho;\lambda)\phi_1(s;\lambda)\gamma_2(\rho,s;\lambda)}{1-2\lambda} \\
G_3(\rho,s;\lambda)&=1_{\R_+}(s-\rho)[1-\chi(\rho\langle\omega\rangle)]s^2(1-s^2)^{-\frac12+\lambda}
\frac{\tilde \phi_1(\rho;\lambda)\phi_1(s;\lambda)\gamma_3(\rho,s;\lambda)}{1-2\lambda} \\
G_4(\rho,s;\lambda)&=1_{\R_+}(\rho-s)\chi(s\langle\omega\rangle)s^2(1-s^2)^{-\frac12+\lambda}
\frac{\phi_1(\rho;\lambda)\phi_0(s;\lambda)\gamma_4(\rho,s;\lambda)}{1-2\lambda} \\
G_5(\rho,s;\lambda)&=1_{\R_+}(\rho-s)[1-\chi(s\langle\omega\rangle)]s^2(1-s^2)^{-\frac12+\lambda}
\frac{\phi_1(\rho;\lambda)\phi_1(s;\lambda)\gamma_5(\rho,s;\lambda)}{1-2\lambda} \\
G_6(\rho,s;\lambda)&=1_{\R_+}(\rho-s)[1-\chi(s\langle\omega\rangle)]s^2(1-s^2)^{-\frac12+\lambda}
\frac{\phi_1(\rho;\lambda)\tilde \phi_1(s;\lambda)\gamma_6(\rho,s;\lambda)}{1-2\lambda}
\end{align*}
and 
\begin{align*}
\gamma_n(\rho,s;\lambda)=\O(\rho^0 s^0)\Oo(\langle\omega\rangle^{-1})+\O((1-\rho)^0s^0\langle\omega\rangle^{-2})
+\O(\rho^0(1-s)\langle\omega\rangle^{-2})
\end{align*}
for all $\rho,s\in (0,1)$, $\epsilon\in [0,\frac13]$, $\omega\in \R$, $n\in \{1,2,\dots,6\}$, and $\lambda=\epsilon+\I\omega$.
\end{lemma}

\begin{proof}
From Lemmas \ref{lem:v0v1}, \ref{lem:v0}, and \ref{lem:v1} we infer
\begin{align*} 
u_0(\rho;\lambda)&=\chi(\rho\langle\omega\rangle)u_0(\rho;\lambda)
+[1-\chi(\rho\langle\omega\rangle)]u_0(\rho;\lambda) \\
&=\chi(\rho\langle\omega\rangle)\phi_0(\rho;\lambda)[1+\O(\rho^0\langle\omega\rangle^{-2})] \\
&\quad +[1-\chi(\rho\langle\omega\rangle)]\phi_1(\rho;\lambda)[1+\O(\rho^0)\Oo(\langle\omega\rangle^{-1})
+\O(\langle\omega\rangle^{-2})+\O((1-\rho)\langle\omega\rangle^{-2})] \\
&\quad -[1-\chi(\rho\langle\omega\rangle)]\tilde\phi_1(\rho;\lambda)
[1+\O(\rho^0)\Oo(\langle\omega\rangle^{-1})
+\O(\langle\omega\rangle^{-2})+\O((1-\rho)\langle\omega\rangle^{-2})] \\
&=\phi_0(\rho;\lambda)+\chi(\rho\langle\omega\rangle)\phi_0(\rho;\lambda)\O(\rho^0\langle\omega\rangle^{-2}) \\
&\quad +[1-\chi(\rho\langle\omega\rangle)]\phi_1(\rho;\lambda)[\O(\rho^0)\Oo(\langle\omega\rangle^{-1})
+\O(\langle\omega\rangle^{-2})+\O((1-\rho)\langle\omega\rangle^{-2})] \\
&\quad -[1-\chi(\rho\langle\omega\rangle)]\tilde\phi_1(\rho;\lambda)
[\O(\rho^0)\Oo(\langle\omega\rangle^{-1})
+\O(\langle\omega\rangle^{-2})+\O((1-\rho)\langle\omega\rangle^{-2})] 
\end{align*}
and the stated decomposition follows from Lemma \ref{lem:v1} and Corollary \ref{cor:W}.
\end{proof}

\subsection{Representation of the semigroup}
As a consequence of Lemma \ref{lem:decomp}, we infer from Eq.~\eqref{eq:SG} the representation
\begin{equation}
\label{eq:decomp} [\mb S(\tau)\tilde{\mb f}]_1(\rho)=[\mb S_0(\tau)\tilde{\mb f}]_1(\rho)
+\frac{1}{2\pi \I}\sum_{n=1}^6
\lim_{N\to\infty}\int_{\epsilon-\I N}^{\epsilon+\I N}e^{\lambda\tau}\int_0^1 G_n(\rho,s;\lambda)F_\lambda(s)\d s 
\d\lambda
\end{equation}
for any $\epsilon>0$ and 
$\tilde{\mb f}=(\tilde f_1,\tilde f_2)\in \rg(\mb I-\mb P)\cap C^2\times C^1([0,1])$, where
\[ F_\lambda(s)=s\tilde f_1'(s)+(\lambda+\tfrac32)\tilde f_1(s)+\tilde f_2(s). \]
For $f\in C([0,1])$ and $n\in \{1,2\dots,6\}$, we define the operators
\begin{align}
\label{def:Teps}
 T_{n,\epsilon}(\tau)f(\rho):&=\frac{1}{2\pi \I}\lim_{N\to\infty}
 \int_{\epsilon-\I N}^{\epsilon+\I N}e^{\lambda\tau}\int_0^1 G_n(\rho,s;\lambda)f(s)\d s \d\lambda \nonumber \\
 &=\frac{1}{2\pi}\int_\R e^{(\epsilon+\I\omega)\tau}\int_0^1 G_n(\rho,s;\epsilon+\I\omega)f(s)\d s\d\omega
\end{align}
for $\tau\geq 0$ and $\rho\in (0,1)$.
Since Eq.~\eqref{eq:decomp} holds for any $\epsilon>0$, we would like to take the limit $\epsilon\to 0+$.
The following result shows that this is indeed possible.

\begin{lemma}
\label{lem:epsto0}
For $\tau\geq 0$, $\rho\in (0,1)$, $f\in C([0,1])$, and $n\in \{1,2,\dots,6\}$, we have
\begin{align*}
 T_n(\tau) f(\rho):=\lim_{\epsilon\to 0+}T_{n,\epsilon}(\tau)f(\rho)
&=\frac{1}{2\pi}\int_\R e^{\I\omega\tau}\int_0^1 G_n(\rho,s;\I\omega)f(s)\d s \d\omega \\
&=\frac{1}{2\pi}\int_0^1 \int_\R e^{\I\omega\tau} G_n(\rho,s;\I\omega)\d\omega f(s)\d s.
\end{align*}
\end{lemma}

\begin{proof}
Since $|\phi_1(\rho;\epsilon+\I\omega)|+|\tilde\phi_1(\rho;\epsilon+\I\omega)|\lesssim \rho^{-1}$ for all
$\rho\in (0,1]$, $\epsilon\in [0,\frac13]$, and $\omega\in \R$, we obtain the bound
$|G_n(\rho,s;\epsilon+\I\omega)|\lesssim \rho^{-1}s(1-s)^{-\frac12+\epsilon}\langle\omega\rangle^{-2}$.
Thus, the claim follows by dominated convergence and Fubini-Tonelli.
\end{proof}

\section{Strichartz estimates}
\noindent We use the formula \eqref{eq:decomp} to derive Strichartz estimates.
More precisely, we prove the following.

\begin{theorem}[Strichartz estimates for $\mb S$]
\label{thm:strich}
Let $\mb S$ be the semigroup generated by $\mb L_0+\mb L'$, see Proposition \ref{prop:SG}.
Furthermore, let $p\in [2,\infty]$ and $q\in [6,\infty]$ such that $\frac{1}{p}+\frac{3}{q}=\frac12$.
Then we have the bound
\[ \|[\mb S(\cdot)(\mb I-\mb P)\mb f]_1\|_{L^p(\R_+)L^q(\B^3)}\lesssim \|(\mb I-\mb P)\mb f\|_{\mc H} \]
for all $\mb f\in \mc H$.
In addition, we have
\[ \left \|\int_0^\tau [\mb S(\tau-\sigma)(\mb I-\mb P)\mb h(\sigma,\cdot)]_1\d\sigma 
\right \|_{L^p_\tau(\R_+)L^q(\B^3)}
\lesssim \|(\mb I-\mb P)\mb h(\tau,\cdot)\|_{L^1_\tau (\R_+)\mc H} \]
for all $\mb h\in C([0,\infty),\mc H)\cap L^1(\R_+,\mc H)$.
\end{theorem}

\subsection{Preparations}

We start with an elementary result on oscillatory integrals which will be applied
frequently in the following.

\begin{lemma}
\label{lem:osc}
We have
\[ \int_\R e^{\I a \omega}[\Oo(\langle\omega\rangle^{-1})+\O(\langle\omega\rangle^{-2})]\d\omega
=O(\langle a \rangle^{-2}) \]
for all $a\in \R\backslash\{0\}$.
\end{lemma}

\begin{proof}
Since 
$\int_\R e^{\I a\omega}\O(\langle\omega\rangle^{-2})\d\omega$
is absolutely convergent, the bound
\[ \int_\R e^{\I a\omega}\O(\langle\omega\rangle^{-2})\d\omega=O(\langle a\rangle^{-2}) \]
follows immediately by means of two integrations by parts.
Furthermore, since $\Oo(\langle\omega\rangle^{-1})$ is odd, we obtain
\[ \int_\R e^{\I a\omega}\Oo(\langle\omega\rangle^{-1})\d\omega
=2\I \int_0^\infty \sin(a\omega)\Oo(\langle\omega\rangle^{-1})\d\omega=:I(a). \]
Without loss of generality we assume $a>0$.
Then we decompose as
\begin{align*}
 I(a)&=\int_0^\infty \chi(a\omega)\sin(a\omega)\Oo(\omega^{-1})\d\omega
 +\int_0^\infty [1-\chi(a\omega)]\sin(a\omega)\Oo(\omega^{-1})\d\omega \\
 &=:I_1(a)+I_2(a).
 \end{align*}
 We have
 \[ I_1(a)=\int_0^\infty \chi(\omega)\sin(\omega)\O(a^0\omega^{-1})\d\omega=O(a^0) \]
 and an integration by parts yields
 \begin{align*}
  I_2(a)&=a^{-1}\int_0^\infty [1-\chi(a\omega)]\cos(a\omega)\O(\omega^{-2})\d\omega
 +\int_0^\infty \chi'(a\omega)\cos(a\omega)\O(\omega^{-1})\d\omega \\
 &=\int_0^\infty [1-\chi(\omega)]\cos(\omega)\O(a^0\omega^{-2})\d\omega
 +\int_0^\infty \chi'(\omega)\cos(\omega)\O(a^0\omega^{-1})\d\omega \\
 &=O(a^0).
\end{align*}
Consequently, we may assume $a\geq 1$ and then
the claim follows by means of two integrations by parts.
\end{proof}

\begin{remark}
As is obvious from the proof, the bound in Lemma \ref{lem:osc} can be improved to $C_N\langle a\rangle^{-N}$ for any $N\in \N$. However, we do not need this.
\end{remark}

\subsection{Kernel bounds}
Next, we prove pointwise bounds on the kernels of the operators $T_n(\tau)$ from 
Lemma \ref{lem:epsto0}.

\begin{proposition}
\label{prop:Gn}
We have the bounds
\[ \left |\int_\R e^{\I\omega\tau}G_n(\rho,s;\I\omega)\d\omega\right |\lesssim s(1-s)^{-\frac12}
\langle \tau+\log(1-s)\rangle^{-2} \]
for all $\tau\geq 0$, $\rho,s\in (0,1)$, and $n\in \{1,2,\dots,6\}$.
\end{proposition}

\begin{proof}
We start with $G_1$
and use
\begin{equation}
\label{eq:phi0decomp} \phi_0(\rho;\I\omega)=\tfrac12 (1-2\I\omega)\int_0^1 [(1+\rho t)^{-\frac12-\I\omega}+(1-\rho t)^{-\frac12-\I\omega}]
\d t 
\end{equation}
in order to get rid of the singular factor $\rho^{-1}$.
Thanks to the cut-off $\chi(\rho\langle\omega\rangle)$ (and the fact that $\rho\in (0,1)$ is fixed),
we may interchange the order of integration and by Lemma \ref{lem:decomp} it suffices to prove the stated
bound for the expression
\begin{align*}
 I_\pm(\rho,s, t;\tau)&:=\int_\R \chi(\rho\langle\omega\rangle)e^{\I\omega\tau}(1\pm \rho t)^{-\frac12-\I\omega}s(1-s)^{-\frac12+\I\omega} \\
&\quad \times [\O(\rho^0 s^0)\Oo(\langle\omega\rangle^{-1})+\O(\rho^0(1-\rho)^0s^0(1-s)^0\langle\omega\rangle^{-2})]
\d\omega.
\end{align*}
The latter may be rewritten as
\begin{align*} I_\pm(\rho,s,t;\tau)&=(1\pm\rho t)^{-\frac12}s(1-s)^{-\frac12}\int_\R \chi(\rho\langle\omega\rangle)
e^{\I\omega[\tau+\log(1-s)-\log(1\pm \rho t)]} \\
&\quad \times [\O(\rho^0 s^0)\Oo(\langle\omega\rangle^{-1})+\O(\rho^0(1-\rho)^0s^0(1-s)^0\langle\omega\rangle^{-2})]
\d\omega.
\end{align*}
On the support of the cut-off $\chi(\rho\langle\omega\rangle)$ we have
$|\log(1\pm \rho t)|\lesssim 1$ for all $t\in [0,1]$ and thus, Lemma \ref{lem:osc} yields the bound
\begin{align*}
 |I_\pm(\rho,s,t;\tau)|&\lesssim s(1-s)^{-\frac12}\langle\tau+\log(1-s)-\log(1\pm \rho t)\rangle^{-2} \\
 &\lesssim s(1-s)^{-\frac12}\langle \tau+\log(1-s)\rangle^{-2}
 \end{align*}
 for all $\tau\geq 0$ and $\rho,s,t\in (0,1)$.

For $G_2$ the relevant expression is
\begin{align*}
 I(\rho,s;\tau)&:=\rho^{-1}(1+\rho)^\frac12s(1-s)^{-\frac12}\int_\R
[1-\chi(\rho\langle\omega\rangle)]e^{\I\omega[\tau+\log(1-s)-\log(1+\rho)]} \\
&\quad\times \O(\rho^0(1-\rho)^0
s^0(1-s)^0\langle\omega\rangle^{-2})\d\omega.
\end{align*}
Note that
\begin{align*}
 \rho^{-1}\int_\R [1-\chi(\rho\langle\omega\rangle)]\O(\langle\omega\rangle^{-2})\d\omega 
&=\rho^{-1}\int_\R \chi(|\omega|)[1-\chi(\rho\langle\omega\rangle)]\O(\langle\omega\rangle^{-2})\d\omega \\
&\quad +\rho^{-1}\int_\R [1-\chi(|\omega|)][1-\chi(\rho|\omega|)]\O(|\omega|^{-2})\d\omega \\
&=\int_\R\chi(|\omega|)[1-\chi(\rho\langle\omega\rangle)]\O(\langle\omega\rangle^{-1})\d\omega \\
&\quad+ \int_\R [1-\chi(|\tfrac{\omega}{\rho}|)][1-\chi(|\omega|)]\O(\rho^0|\omega|^{-2})\d\omega \\
&=O(\rho^0)
\end{align*}
and thus, two integrations by parts yield
\begin{align*}
 |I(\rho,s;\tau)|&\lesssim s(1-s)^{-\frac12}\langle \tau+\log(1-s)-\log(1+\rho)\rangle^{-2}  \\
 &\lesssim s(1-s)^{-\frac12}\langle \tau+\log(1-s)\rangle^{-2}.
 \end{align*}
 
The bound for $G_3$ is similar, i.e., here we have to consider
\begin{align*}
 I(\rho,s;\tau)&:=\rho^{-1}(1-\rho)^\frac12s(1-s)^{-\frac12}\int_\R
[1-\chi(\rho\langle\omega\rangle)]e^{\I\omega[\tau+\log(1-s)-\log(1-\rho)]} \\
&\quad\times \O(\rho^0(1-\rho)^0
s^0(1-s)^0\langle\omega\rangle^{-2})\d\omega
\end{align*}
and the argument from above yields 
\[ |I(\rho,s;\tau)|\lesssim (1-\rho)^\frac12 s(1-s)^{-\frac12}
\langle \tau+\log(1-s)-\log(1-\rho)\rangle^{-2}. \]
Since $(1-\rho)^\frac12 \langle \tau+\log(1-s)-\log(1-\rho)\rangle^{-2}\lesssim \langle
\tau+\log(1-s)\rangle^{-2}$, the desired bound follows.

For $G_4$ we use again the representation \eqref{eq:phi0decomp} which leads to
\begin{align*} I_\pm(\rho,s,t;\tau)&=1_{\R_+}(\rho-s)\rho^{-1}(1+\rho)^\frac12
s^2(1-s^2)^{-\frac12}(1\pm s t)^{-\frac12} \\
&\quad \times\int_\R \chi(s\langle\omega\rangle)
e^{\I\omega[\tau+\log(1-s^2)-\log(1\pm s t)-\log(1+\rho)]} \\
&\quad \times [\O(\rho^0 s^0)\Oo(\langle\omega\rangle^{-1})+\O(\rho^0(1-\rho)^0s^0(1-s)^0\langle\omega\rangle^{-2})]
\d\omega.
\end{align*}
and from Lemma \ref{lem:osc} we obtain the desired bound
\[ |I_\pm(\rho,s,t;\tau)|\lesssim s(1-s)^{-\frac12}\langle \tau+\log(1-s)\rangle^{-2}. \]

In view of the bound $1_{\R_+}(\rho-s)\rho^{-1}\lesssim s^{-1}$, the estimate for 
$G_5$ is identical to the one for $G_2$.
Finally, for $G_6$ we consider
\begin{align*}
I(\rho,s;\tau)&:=1_{\R_+}(\rho-s)\rho^{-1}(1+\rho)^\frac12 s (1+s)^{-\frac12}
\int_\R [1-\chi(s\langle\omega\rangle)]e^{\I\omega[\tau+\log(1+s)-\log(1+\rho)]} \\
&\quad\times \O(\rho^0(1-\rho)^0
s^0(1-s)^0\langle\omega\rangle^{-2})\d\omega
\end{align*}
and as for $G_3$ we find
\begin{align*} 
|I(\rho,s;\lambda)|&\lesssim s(1+s)^{-\frac12}\langle \tau+\log(1+s)-\log(1+\rho)\rangle^{-2} 
\lesssim s\langle\tau\rangle^{-2} \\
&\lesssim s(1-s)^{-\frac12}\langle \tau+\log(1-s)\rangle^{-2}.
\end{align*}
\end{proof}

The pointwise bounds from Proposition \ref{prop:Gn} are sufficient to obtain Strichartz
estimates for the operators $T_n$.

\begin{lemma}
\label{lem:Tn}
Let $p\in [2,\infty]$, $q\in [6,\infty]$ and assume $\frac{1}{p}+\frac{3}{q}=\frac12$.
Then we have
\[ \|T_n(\cdot)f\|_{L^p(\R_+)L^q(\B^3)}\lesssim \|f\|_{L^2(\B^3)} \]
for all $f\in C([0,1])$ and $n\in \{1,2,\dots,6\}$.
\end{lemma}

\begin{proof}
We set
\[ K_n(\rho,s;\tau):=\frac{1}{2\pi}\int_\R e^{\I\omega\tau}G_n(\rho,s;\I\omega)\d\omega \]
and from Lemma \ref{lem:epsto0} we infer
\[ T_n(\tau)f(\rho)=\int_0^1 K_n(\rho,s;\tau)f(s)\d s. \]
The change of variable $s=1-e^{-y}$ yields
\[ T_n(\tau)f(\rho)=\int_0^\infty e^{-y/2}K_n(\rho,1-e^{-y};\tau)f(1-e^{-y})e^{-y/2}\d y \]
and from Proposition \ref{prop:Gn} we obtain the bound
\[ \|T_n(\tau)f\|_{L^\infty(\B^3)}\lesssim \int_0^\infty \langle \tau-y\rangle^{-2}|f(1-e^{-y})|
(1-e^{-y})e^{-y/2}\d y. \]
Consequently, Young's inequality implies
\begin{align*}
 \|T_n(\cdot)f\|_{L^2(\R_+)L^\infty(\B^3)}^2&\lesssim \int_0^\infty
|f(1-e^{-y})|^2(1-e^{-y})^2 e^{-y}\d y \\
&=\int_0^1 |f(s)|^2 s^2 \d s\simeq \|f\|_{L^2(\B^3)}^2.
\end{align*}
Furthermore, by Cauchy-Schwarz we infer
\[ \|T_n(\tau)f\|_{L^\infty(\B^3)}\lesssim \|\langle \tau-\cdot\rangle^{-2}\|_{L^2(\R)}
\|f\|_{L^2(\B^3)} \lesssim \|f\|_{L^2(\B^3)}\]
which yields
$\|T_n(\cdot)f\|_{L^\infty(\R_+)L^6(\B^3)}\lesssim \|f\|_{L^2(\B^3)}$.
Thus, the claim follows by interpolation.
\end{proof}

\subsection{The operators $\dot T_n(\tau)$}
The operators $T_n(\tau)$ alone are not sufficient to prove Theorem \ref{thm:strich} since the
function $F_\lambda$ in Eq.~\eqref{eq:decomp} contains a term $\lambda \tilde f_1$.
Formally, multiplication by a factor $\lambda$ is equivalent to taking a derivative with
respect to $\tau$. 
Consequently, for $\tau\geq 0$, $\rho\in (0,1)$, $f\in C^1([0,1])$, and $n\in \{1,2,\dots,6\}$, we define
\begin{equation}
\label{def:dotT}
\dot T_{n,\epsilon}(\tau)f(\rho):=\frac{1}{2\pi \I}\lim_{N\to\infty}\int_{\epsilon-\I N}^{\epsilon+\I N}
\lambda e^{\lambda\tau}\int_0^1 G_n(\rho,s;\lambda)f(s)\d s\d\lambda.
\end{equation}
Note that the additional factor of $\lambda$ spoils the absolute convergence of the $\lambda$-integral.
Thus, we perform an integration by parts with respect to $s$ in order to gain a factor $\langle\lambda\rangle^{-1}$.
This illustrates the general philosophy that one may trade derivatives in $s$ for decay in $\lambda$.
 
\begin{proposition}
\label{prop:dotTn}
Let $p\in [2,\infty]$, $q\in [6,\infty]$ and assume $\frac{1}{p}+\frac{3}{q}=\frac12$. 
Furthermore, for $n\in \{1,2,\dots,6\}$, set
\[ \dot T_n(\tau)f(\rho):=\lim_{\epsilon\to 0+}\dot T_{n,\epsilon}(\tau)f(\rho). \]
Then we have
\[ \|\dot T_n(\cdot)f\|_{L^p(\R_+)L^q(\B^3)}\lesssim \|f\|_{H^1(\B^3)} \]
for all $f\in C^1([0,1])$.
\end{proposition}

\begin{proof}
The $s$-dependent part of $G_n(\rho,s;\lambda)$, for $n\in \{1,2,3\}$, is given by
\[ 1_{\R_+}(s-\rho)s(1-s)^{-\frac12+\lambda}\gamma_n(\rho,s;\lambda). \]
We have
\begin{align*}
\int_0^1 &1_{\R_+}(s-\rho)(1-s)^{-\frac12+\lambda}s\gamma_n(\rho,s;\lambda) f(s)\d s \\
&=-\frac{2}{1+2\lambda}\int_\rho^1 \partial_s (1-s)^{\frac12+\lambda}s\gamma_n(\rho,s;\lambda)f(s)\d s \\
&=\frac{2}{1+2\lambda}(1-\rho)^{\frac12+\lambda}\rho \gamma_n(\rho,\rho;\lambda)f(\rho)
+\frac{2}{1+2\lambda}\int_\rho^1 (1-s)^{\frac12+\lambda}\partial_s[s\gamma_n(\rho,s;\lambda)f(s)]\d s.
\end{align*}
From Lemma \ref{lem:decomp} we recall
\[ \gamma_n(\rho,s;\lambda)=\O(\rho^0 s^0)\Oo(\langle\omega\rangle^{-1})
+\O(\rho^0(1-\rho)^0s^0(1-s)^0\langle\omega\rangle^{-2}) \]
and thus, $s(1-s)\partial_s \gamma_n(\rho,s;\lambda)$ is of the same form as $\gamma_n(\rho,s;\lambda)$.
Furthermore, $\|f\|_{L^2(0,1)}\lesssim \|f\|_{H^1(\B^3)}$ and therefore, it suffices to consider the operators
induced by the boundary term, i.e.,
\begin{align*} 
B_{1,\epsilon}(\tau)f(\rho)&:=\frac{2\rho f(\rho)}{2\pi \I}\lim_{N\to\infty}\int_{\epsilon-\I N}^{\epsilon+ \I N}
\frac{\lambda}{1-4\lambda^2} e^{\lambda\tau}\chi(\rho\langle\omega\rangle)\phi_0(\rho;\lambda)
(1-\rho)^{\frac12+\lambda}\gamma_1(\rho,\rho;\lambda)\d\lambda \\
B_{2,\epsilon}(\tau)f(\rho)&:=\frac{2\rho f(\rho)}{2\pi \I}\lim_{N\to\infty}\int_{\epsilon-\I N}^{\epsilon+ \I N}
\frac{\lambda}{1-4\lambda^2} e^{\lambda\tau} \\
&\quad \times [1-\chi(\rho\langle\omega\rangle)]\phi_1(\rho;\lambda)
(1-\rho)^{\frac12+\lambda}\gamma_2(\rho,\rho;\lambda)\d\lambda \\
B_{3,\epsilon}(\tau)f(\rho)&:=\frac{2\rho f(\rho)}{2\pi \I}\lim_{N\to\infty}\int_{\epsilon-\I N}^{\epsilon+ \I N}
\frac{\lambda}{1-4\lambda^2} e^{\lambda\tau} \\
&\quad \times [1-\chi(\rho\langle\omega\rangle)]\tilde \phi_1(\rho;\lambda)
(1-\rho)^{\frac12+\lambda}\gamma_3(\rho,\rho;\lambda)\d\lambda
\end{align*}
where $\lambda=\epsilon+\I\omega$.
The integrands are bounded by $C\rho^{-1}\langle\omega\rangle^{-2}$ and thus, the limits
\[ B_n(\tau)f(\rho):=\lim_{\epsilon\to 0+}B_{n,\epsilon}(\tau)f(\rho) \] exist for $n\in \{1,2,3\}$.
Explicitly, we have
\begin{align*}
 B_1(\tau)f(\rho)=\frac{\rho f(\rho)}{\pi}\int_\R \frac{\I\omega}{1+4\omega^2}e^{\I\omega\tau}
 \chi(\rho\langle\omega\rangle)\phi_0(\rho;\I\omega)
(1-\rho)^{\frac12+\I\omega}\gamma_1(\rho,\rho;\I\omega)\d\omega
 \end{align*}
 and by using Eq.~\eqref{eq:phi0decomp}, we infer from Lemma \ref{lem:osc} the bound
 \[ |B_1(\tau)f(\rho)|\lesssim \langle\tau\rangle^{-2}|\rho f(\rho)|\lesssim \langle\tau\rangle^{-2}\|f\|_{H^1(\B^3)}, \]
 cf.~the proof of Proposition \ref{prop:Gn}.
Similarly, for $B_2$ we obtain
\begin{align*} 
B_2(\tau)f(\rho)
&=\frac{\rho f(\rho)}{\pi}\int_\R e^{\I\omega\tau}
 [1-\chi(\rho\langle\omega\rangle)]\rho^{-1}(1+\rho)^{\frac12-\I\omega} \\
&\quad\times (1-\rho)^{\frac12+\I\omega}\O(\rho^0(1-\rho)^0\langle\omega\rangle^{-2})\d\omega 
\end{align*}
and as in the estimate for $G_2$ in the proof of Proposition \ref{prop:Gn}, two integrations by parts yield
\[ |B_2(\tau)f(\rho)|\lesssim (1-\rho)^\frac12 \langle\tau+\log(1-\rho)\rangle^{-2}|\rho f(\rho)|
\lesssim \langle\tau\rangle^{-2}\|f\|_{H^1(\B^3)}. \]
For $B_3$ we infer
\[ B_3(\tau)f(\rho)=\frac{\rho f(\rho)}{\pi}\int_\R e^{\I\omega\tau}
 [1-\chi(\rho\langle\omega\rangle)]\rho^{-1}(1-\rho)\O(\rho^0(1-\rho)^0\langle\omega\rangle^{-2})\d\omega \]
 and as before, we obtain $|B_3(\tau)f(\rho)|\lesssim \langle\tau\rangle^{-2}\|f\|_{H^1(\B^3)}$.
 
 Next, we turn to the cases $n\in \{4,5,6\}$.
 The $s$-dependent part of $G_4$ is given by
 \[ 1_{\R_+}(\rho-s)\chi(s\langle\omega\rangle)s[(1-s)^{-\frac12+\lambda}-(1+s)^{-\frac12+\lambda}]
 \gamma_4(\rho,s;\lambda). \]
 An integration by parts yields
 \begin{align*}
 \int_0^1 &1_{\R_+}(\rho-s)\chi(s\langle\omega\rangle)s[(1-s)^{-\frac12+\lambda}-(1+s)^{-\frac12+\lambda}]
 \gamma_4(\rho,s;\lambda) f(s)\d s \\
 &=-\frac{2}{1+2\lambda}\int_0^\rho \partial_s[(1-s)^{\frac12+\lambda}+(1+s)^{\frac12+\lambda}-2]
 \chi(s\langle\omega\rangle)s\gamma_4(\rho,s;\lambda)f(s)\d s \\
 &=-\frac{2}{1+2\lambda}[(1-\rho)^{\frac12+\lambda}+(1+\rho)^{\frac12+\lambda}-2]\chi(\rho\langle\omega\rangle)
 \gamma_4(\rho,\rho;\lambda)\rho f(\rho) \\
 &\quad +\frac{2}{1+2\lambda}\int_0^\rho [(1-s)^{\frac12+\lambda}+(1+s)^{\frac12+\lambda}-2]
 \partial_s[\chi(s\langle\omega\rangle)s\gamma_4(\rho,s;\lambda)f(s)]\d s.
 \end{align*}
 Recall from Lemma \ref{lem:decomp} that
 \[ \gamma_4(\rho,s;\lambda)=\O(\rho^0s^0)\Oo(\langle\omega\rangle^{-1})+\O((1-\rho)^0s^0\langle\omega\rangle^{-2})
 +\O(\rho^0 (1-s)\langle\omega\rangle^{-2}) \]
 and thus, $s\partial_s\gamma_4(\rho,s;\lambda)$ is of the same form as $\gamma_4(\rho,s;\lambda)$.
Consequently, since
 \begin{equation}
 \label{eq:decomp2}
  (1-s)^{\frac12+\lambda}+(1+s)^{\frac12+\lambda}-2=-\tfrac12(1+2\lambda)s\int_0^1 [(1-st)^{-\frac12+\lambda}
 -(1+st)^{-\frac12+\lambda}]\d t 
 \end{equation}
 and $\|f\|_{L^2(0,1)}\lesssim \|f\|_{H^1(\B^3)}$, the integral term leads to operators
 that can be handled as $T_4(\tau)$.
Thus, it suffices to consider the operator generated by the boundary term, i.e.,
\begin{align*} 
B_{4,\epsilon}(\tau)f(\rho)&:=\frac{2\rho f(\rho)}{2\pi \I}\lim_{N\to\infty}
\int_{\epsilon-\I N}^{\epsilon+\I N} \frac{\lambda}{1-4\lambda^2}e^{\lambda\tau} \\
&\quad \times \chi(\rho\langle\omega\rangle)[(1-\rho)^{\frac12+\lambda}+(1+\rho)^{\frac12+\lambda}-2]\phi_1(\rho;\lambda)
 \gamma_4(\rho,\rho;\lambda)\d\lambda.
\end{align*}
Taking the limit $\epsilon\to 0+$
yields the operator
\begin{align*} 
B_4(\tau)f(\rho)&:=\lim_{\epsilon\to 0+}B_{4,\epsilon}(\tau)f(\rho)=\frac{\rho f(\rho)}{\pi}\int_\R
\frac{\I\omega}{1+4\omega^2}e^{\I\omega\tau} \\
&\times \chi(\rho\langle\omega\rangle)[(1-\rho)^{\frac12+\I\omega}+(1+\rho)^{\frac12+\I\omega}-2]\phi_1(\rho;\I\omega)
 \gamma_4(\rho,\rho;\I\omega)\d\omega
\end{align*}
and by using Eq.~\eqref{eq:decomp2}, Lemma \ref{lem:osc} yields the bound
\[ |B_4(\tau)f(\rho)|\lesssim \langle\tau\rangle^{-2}|\rho f(\rho)|\lesssim \langle\tau\rangle^{-2}\|f\|_{H^1(\B^3)}. \]
For $n\in \{5,6\}$, the $s$-dependent part of $G_n$ reads
\[ 1_{\R_+}(\rho-s)[1-\chi(s\langle\omega\rangle)]s(1\pm s)^{-\frac12+\lambda}\gamma_n(\rho,s;\lambda) \]
and we proceed analogously to the above by an integration by parts to obtain the 
desired bound.
\end{proof}

\subsection{Proof of Theorem \ref{thm:strich}}
From Eq.~\eqref{eq:decomp} we obtain
\[ [\mb S(\tau)\tilde{\mb f}]_1=[\mb S_0(\tau)\tilde{\mb f}]_1
+\sum_{n=1}^6 \left [ T_n(\tau)(|\cdot|\tilde f_1'+\tfrac32 \tilde f_1+\tilde f_2)
+\dot T_n(\tau)\tilde f_1 \right ] \]
and from Propositions \ref{prop:strich}, \ref{prop:dotTn} and Lemma \ref{lem:Tn} we obtain the
homogeneous Strichartz estimates of Theorem \ref{thm:strich}, provided $\tilde{\mb f}\in C^2\times C^1([0,1])$.
The inhomogeneous Strichartz estimates can be deduced from the homogeneous ones by
the same argument as in the proof of Proposition \ref{prop:strich}.
Thus, a density argument finishes the proof of Theorem \ref{thm:strich}.

\section{Improved energy bound}

\noindent We also need to improve the energy estimate from Proposition \ref{prop:SG}.
Our goal is to show that the constant $C_\epsilon$ is in fact independent of $\epsilon$
which allows us to choose $\epsilon=0$.

\subsection{Preliminaries}

In the sequel we will encounter the following problem: We would like to interchange
a derivative or a limit with an integral but the resulting expression does not converge absolutely.
Consequently, the dominated convergence theorem does not apply.
However, instead of proving a general result that covers each and every case that occurs, 
we rather illustrate the logic on a typical example.

\begin{lemma}
\label{lem:oscder}
Let $f(\omega)=\O(\langle\omega\rangle^{-2})$. Then we have
\[ \partial_a \int_\R e^{\I a \omega}f(\omega)\d\omega=
\I \int_\R \omega e^{\I a\omega}f(\omega)\d\omega
 \]
for all $a\in \R\backslash \{0\}$.
\end{lemma}

\begin{proof}
An integration by parts yields
\[ \int_\R e^{\I a\omega}f(\omega)\d\omega=-\frac{1}{\I a}\int_\R e^{\I a\omega}f'(\omega)\d\omega. \]
The integrand of the latter integral decays like $\langle\omega\rangle^{-3}$ and thus,
by dominated convergence we infer
\begin{align*}
 \partial_a \int_\R e^{\I a\omega}f(\omega)\d\omega&=\frac{1}{\I a^2}\int_\R e^{\I a\omega}f'(\omega)\d\omega
-\frac{1}{a}\int_\R \omega e^{\I a\omega}f'(\omega)\d\omega \\
&=\I\int_\R \omega e^{\I a\omega}f(\omega)\d\omega.
\end{align*}
\end{proof}

\subsection{Decomposition}
Now let $\tilde{\mb f}\in \rg(\mb I-\mb P)\cap C^2\times C^1([0,1])$.
From Eq.~\eqref{eq:SG} and (a suitable variant of) Lemma \ref{lem:oscder}, we obtain the representation
\begin{equation}
\label{eq:SG'1} \partial_\rho [\mb S(\tau)\tilde{\mb f}]_1(\rho)=\frac{1}{2\pi \I}\lim_{N\to\infty}
\int_{\epsilon-\I N}^{\epsilon+\I N}e^{\lambda\tau}\int_0^1 G'(\rho,s;\lambda)F_\lambda(s)\d\lambda, 
\end{equation}
where
\begin{equation} 
\label{eq:G'}
G'(\rho,s;\lambda):=\frac{s^2(1-s^2)^{-\frac12+\lambda}}{(1-2\lambda)w_0(\lambda)}
\left \{ 
\begin{array}{l}
\partial_\rho u_0(\rho;\lambda)u_1(s;\lambda)\mbox{ if }\rho\leq s \\
\partial_\rho u_1(\rho;\lambda)u_0(s;\lambda)\mbox{ if }\rho \geq s
\end{array} \right . 
\end{equation} 
and, as always, $F_\lambda(s)=s\tilde f_1'(s)+(\lambda+\frac32)\tilde f_1(s)+\tilde f_2(s)$.
To be more precise, one needs to perform an integration by parts with respect to $s$, thereby
exploiting the assumed differentiability of $\tilde{\mb f}$, in order to generate 
a factor $\langle\lambda\rangle^{-1}$ which yields enough
decay to make sense of the above representation formula.

As in Lemma \ref{lem:decomp}, we peel off the free part
\[ G_0'(\rho,s;\lambda):=\frac{s^2(1-s^2)^{-\frac12+\lambda}}{1-2\lambda}
\left \{ 
\begin{array}{l}
\partial_\rho \phi_0(\rho;\lambda)\phi_1(s;\lambda)\mbox{ if }\rho\leq s \\
\partial_\rho \phi_1(\rho;\lambda)\phi_0(s;\lambda)\mbox{ if }\rho \geq s
\end{array} \right . .
\]
Note that the contribution from $G_0'$ is already handled by the abstract theory, see Proposition \ref{prop:gen}.

\begin{lemma}
\label{lem:decompG'}
We have the decomposition 
\[ G'(\rho,s;\lambda)=G_0'(\rho,s;\lambda)+\sum_{n=1}^6 \rho^{-1}
\tilde G_n(\rho,s;\lambda)+\sum_{n=1}^4 G_n'(\rho,s;\lambda) \]
where
\begin{align*}
G_1'(\rho,s;\lambda)&=1_{\R_+}(s-\rho)s^2(1-s^2)^{-\frac12+\lambda}
\phi_1(\rho;\lambda)\phi_1(s;\lambda)\gamma_1'(\rho,s;\lambda) \\
G_2'(\rho,s;\lambda)&=1_{\R_+}(s-\rho)s^2(1-s^2)^{-\frac12+\lambda}(1-\rho)^{-1}
\tilde \phi_1(\rho;\lambda)\phi_1(s;\lambda)\gamma_2'(\rho,s;\lambda) \\
G_3'(\rho,s;\lambda)&=1_{\R_+}(\rho-s)s^2(1-s^2)^{-\frac12+\lambda}
\phi_1(\rho;\lambda)\phi_1(s;\lambda)\gamma_3'(\rho,s;\lambda) \\
G_4'(\rho,s;\lambda)&=1_{\R_+}(\rho-s)s^2(1-s^2)^{-\frac12+\lambda}
\phi_1(\rho;\lambda)\tilde \phi_1(s;\lambda)\gamma_4'(\rho,s;\lambda) 
\end{align*}
with 
\[ \gamma_n'(\rho,s;\lambda)=\O(\rho^0 s^0)\Oo(\langle\omega\rangle^{-1})
+\O((1-\rho)^0s^0\langle\omega\rangle^{-2})+\O(\rho^0(1-s)\langle\omega\rangle^{-2}) \]
and $\tilde G_n$ is of the same form as $G_n$ from Lemma \ref{lem:decomp}. 
\end{lemma}

\begin{proof}
By Lemma \ref{lem:v1} and Eq.~\eqref{eq:utov} we have 
\[ u_1(\rho;\lambda)=\phi_1(\rho;\lambda)[1+a(\rho;\lambda)] \] where 
$a(\rho;\lambda)=a_1(\rho)\Oo(\langle\omega\rangle^{-1})+\O((1-\rho)\langle\omega\rangle^{-2})$
and $\phi_1(\rho;\lambda)=\rho^{-1}(1+\rho)^{\frac12-\lambda}$.
Consequently, we infer
\begin{align*}
\partial_\rho u_1(\rho;\lambda)&=\partial_\rho \phi_1(\rho;\lambda)[1+a(\rho;\lambda)]
+\phi_1(\rho;\lambda)\partial_\rho a(\rho;\lambda) \\
&=\partial_\rho\phi_1(\rho;\lambda)-\rho^{-1}\phi_1(\rho;\lambda)a(\rho;\lambda)
+(\tfrac12-\lambda)(1+\rho)^{-1}\phi_1(\rho;\lambda)a(\rho;\lambda) \\
&\quad +\phi_1(\rho;\lambda)\partial_\rho a(\rho;\lambda) \\
&=\partial_\rho \phi_1(\rho;\lambda)+\rho^{-1}\phi_1(\rho;\lambda)[\O(\rho^0)\Oo(\langle\omega\rangle^{-1})
+\O((1-\rho)^0\langle\omega\rangle^{-2})] \\
&\quad +\tfrac12(1-2\lambda)\phi_1(\rho;\lambda)[\O(\rho^0)\Oo(\langle\omega\rangle^{-1})
+\O((1-\rho)^0\langle\omega\rangle^{-2})].
\end{align*}
Analogously, by Lemma \ref{lem:v0},
\begin{align*}
\chi(\rho\langle\omega\rangle)\partial_\rho u_0(\rho;\lambda)&=
\chi(\rho\langle\omega\rangle)\left [\partial_\rho\phi_0(\rho;\lambda)
+\partial_\rho\phi_0(\rho;\lambda)\O(\rho^2\langle\omega\rangle^0)+\phi_0(\rho;\lambda)\O(\rho\langle\omega\rangle^0)
\right ] 
\end{align*}
and by using
\[ \partial_\rho \phi_0(\rho;\lambda)=-\rho^{-1}\phi_0(\rho;\lambda)+\tfrac12(1-2\lambda)
[(1+\rho)^{-1}\phi_1(\rho;\lambda)+(1-\rho)^{-1}\tilde\phi_1(\rho;\lambda)], \]
we find
\begin{align*}
\chi(\rho\langle\omega\rangle)\partial_\rho u_0(\rho;\lambda)&=
\chi(\rho\langle\omega\rangle)\partial_\rho \phi_0(\rho;\lambda)
+\chi(\rho\langle\omega\rangle)\rho^{-1}\phi_0(\rho;\lambda)\O(\rho^0\langle\omega\rangle^{-2}) \\
&\quad +\chi(\rho\langle\omega\rangle)(1-2\lambda)\phi_1(\rho;\lambda)\O(\rho^0\langle\omega\rangle^{-2}) \\
&\quad +\chi(\rho\langle\omega\rangle)(1-2\lambda)\tilde \phi_1(\rho;\lambda)\O(\rho^0\langle\omega\rangle^{-2}) .
\end{align*}
Finally, by Lemma \ref{lem:v0v1},
\begin{align*}
[1-\chi(\rho\langle\omega\rangle)]\partial_\rho u_0(\rho;\lambda)&=
[1-\chi(\rho\langle\omega\rangle)]\big \{
\partial_\rho \phi_1(\rho;\lambda)[1+b_1(\rho;\lambda)]+\phi_1(\rho;\lambda)\partial_\rho b_1(\rho;\lambda) \big \} \\
&\quad -[1-\chi(\rho\langle\omega\rangle)]\big \{
\partial_\rho \tilde \phi_1(\rho;\lambda)[1+b_2(\rho;\lambda)]+\tilde \phi_1(\rho;\lambda)
\partial_\rho  b_2(\rho;\lambda) \big \}
\end{align*}
where
\begin{align*}
b_j(\rho;\lambda)&=\O(\rho^0)\Oo(\langle\omega\rangle^{-1})+\O(\langle\omega\rangle^{-2})
+\O((1-\rho)\langle\omega\rangle^{-2}),\qquad j\in \{1,2\}. 
\end{align*}
Thus, we obtain
\begin{align*}
[1-\chi(\rho\langle\omega\rangle)]\partial_\rho u_0(\rho;\lambda)&=
[1-\chi(\rho\langle\omega\rangle)]\partial_\rho \phi_0(\rho;\lambda) \\
&\quad +[1-\chi(\rho\langle\omega\rangle)]\rho^{-1}\phi_1(\rho;\lambda)c_1(\rho;\lambda) \\
&\quad +(1-2\lambda)\phi_1(\rho;\lambda)c_2(\rho;\lambda)\\
&\quad +[1-\chi(\rho\langle\omega\rangle)]\rho^{-1}\tilde \phi_1(\rho;\lambda)c_3(\rho;\lambda) \\
&\quad +(1-2\lambda)(1-\rho)^{-1}\tilde \phi_1(\rho;\lambda)c_4(\rho;\lambda)
\end{align*}
where
\[ c_j(\rho;\lambda)=\O(\rho^0)\Oo(\langle\omega\rangle^{-1})+\O((1-\rho)^0\langle\omega\rangle^{-2}),
\qquad j\in \{1,2,3,4\}. \]
This yields the stated decomposition.
\end{proof}

In view of Lemma \ref{lem:decompG'} we define the operators
\begin{align*} 
\tilde S_{n,\epsilon}(\tau)f(\rho):&=\frac{1}{2\pi \I}\lim_{N\to \infty}\int_{\epsilon-\I N}^{\epsilon+\I N}
e^{\lambda\tau}\int_0^1 \rho^{-1}\tilde G_n(\rho,s;\lambda)f(s)\d s\d\lambda \\
&=\frac{1}{2\pi}\int_\R e^{(\epsilon+\I\omega)\tau}\int_0^1 \rho^{-1}\tilde G_n(\rho,s;\epsilon+\I\omega)f(s)
\d s\d\omega,\qquad n\in \{1,2,\dots,6\}
\end{align*}
and
\begin{align*}
S_{n,\epsilon}'(\tau)f(\rho):=\frac{1}{2\pi}\int_\R e^{(\epsilon+\I\omega)\tau}
\int_0^1 G_n'(\rho,s;\epsilon+\I\omega)f(s)
\d s\d\omega,\qquad n\in \{1,2,3,4\}
\end{align*}
for $\tau\geq 0$, $\rho \in (0,1)$, and $f\in C^1([0,1])$.
From Lemma \ref{lem:epsto0} we obtain
\[ \tilde S_n(\tau)f(\rho):=\lim_{\epsilon\to 0+}\tilde S_{n,\epsilon}(\tau)f(\rho)=\frac{1}{2\pi}\int_0^1 \int_\R e^{\I\omega\tau}
\rho^{-1}\tilde G_n(\rho,s;\I\omega)\d\omega f(s)\d s. \]
A similar result is true for $S_{n,\epsilon}'(\tau)$.

\begin{lemma}
\label{lem:epsto0'}
For $\tau\geq 0$, $\rho\in (0,1)$, $f\in C^1([0,1])$, and $n\in \{1,2,3,4\}$, we have
\begin{align*}
S_n'(\tau)f(\rho):=\lim_{\epsilon\to 0}S_{n,\epsilon}'(\tau)f(\rho)&=
\frac{1}{2\pi}\int_\R e^{\I\omega\tau}\int_0^1 G_n'(\rho,s;\I\omega)f(s)\d s\d\omega \\
&=\frac{1}{2\pi}\int_0^1 \int_\R e^{\I\omega\tau}G_n'(\rho,s;\I\omega)\d\omega f(s)\d s.
\end{align*}
\end{lemma}

\begin{proof}
The first equality follows by an argument similar to the proof of Lemma \ref{lem:oscder}, i.e.,
one performs an integration by parts with respect to $s$ in order to gain a factor
of $\langle\omega\rangle^{-1}$, which renders the integral absolutely convergent.
Then one may apply dominated convergence to take the limit $\epsilon\to 0+$.
Finally, one performs another integration by parts in $s$ to remove the derivative
from $f$.
The second equality follows by Fubini and dominated convergence.
\end{proof}

\subsection{Kernel bounds}
We prove pointwise bounds for the kernels of the operators $S_n'(\tau)$.

\begin{lemma}
\label{lem:boundsG'}
We have the bounds
\[ \left |\int_\R e^{\I\omega\tau}G_n'(\rho,s;\I\omega)\d\omega \right |\lesssim \rho^{-1}(1-\rho)^{-\frac12}
s(1-s)^{-\frac12}\langle\tau-\log(1-\rho)+\log(1-s)\rangle^{-2} \]
for all $\tau\geq 0$, $\rho,s\in (0,1)$, and $n\in \{1,2,3,4\}$.
\end{lemma}

\begin{proof}
We have
\[ G_1'(\rho,s;\lambda)=1_{\R_+}(s-\rho)\rho^{-1}(1+\rho)^{\frac12-\lambda}s(1-s)^{-\frac12+\lambda}
\gamma_1'(\rho,s;\lambda) \]
and thus, for $G_1'$ it suffices to estimate
\begin{align*} I_1(\rho,s;\tau):=&\rho^{-1}(1+\rho)^\frac12 s(1-s)^{-\frac12} \int_\R 
e^{\I\omega[\tau-\log(1+\rho)+\log(1-s)]} \\
&\times [\O(\rho^0s^0)\Oo(\langle\omega\rangle^{-1} )
+\O(\rho^0s^0(1-\rho)^0(1-s)^0\langle\omega\rangle^{-2})]\d\omega.
\end{align*}
Lemma \ref{lem:osc} yields
\begin{align*}
 |I_1(\rho,s;\tau)|&\lesssim \rho^{-1}(1+\rho)^\frac12 s(1-s)^{-\frac12}\langle \tau-\log(1+\rho)+\log(1-s)\rangle^{-2} \\
 &\lesssim \rho^{-1}(1-\rho)^{-\frac12}
s(1-s)^{-\frac12}\langle\tau-\log(1-\rho)+\log(1-s)\rangle^{-2}.
 \end{align*}
For $G_2'$ we note that
\[ G_2'(\rho,s;\lambda)=1_{\R_+}(s-\rho)\rho^{-1}(1-\rho)^{-\frac12-\lambda}s(1-s)^{-\frac12+\lambda}
\gamma_2'(\rho,s;\lambda) \]
and Lemma \ref{lem:osc} yields the desired bound.
The estimate for $G_3'$ is the same as the estimate for $G_1'$.
Finally, $G_4'$ is given by
\[ G_4'(\rho,s;\lambda)=1_{\R_+}(\rho-s)\rho^{-1}(1+\rho)^{\frac12-\lambda}s(1+s)^{-\frac12+\lambda}\gamma_4'(\rho,s;\lambda) \]
and since $\langle\tau\rangle^{-2}\lesssim (1-\rho)^{-\frac12}(1-s)^{-\frac12}\langle\tau-\log(1-\rho)+\log(1-s)\rangle^{-2}$,
Lemma \ref{lem:osc} yields the desired bound.
\end{proof}

The bounds from Lemma \ref{lem:boundsG'} easily imply the $L^2$-boundedness of the operators
$S_n'(\tau)$.

\begin{lemma}
\label{lem:SL2}
We have
\begin{align*}
 \|S_n'(\tau)f\|_{L^2(\B^3)}&\lesssim \|f\|_{L^2(\B^3)},\qquad n\in \{1,2,3,4\}  \\
 \|\tilde S_n(\tau)f\|_{L^2(\B^3)}&\lesssim \|f\|_{L^2(\B^3)},\qquad n\in \{1,2,\dots,6\}
 \end{align*}
 for all $\tau\geq 0$ and $f\in C^1([0,1])$.
\end{lemma}

\begin{proof}
We set
\[ K_n'(\rho,s;\tau):=\frac{1}{2\pi}\int_\R e^{\I\omega\tau}G_n'(\rho,s;\I\omega)\d\omega. \]
Thus,
\[ S_n'(\tau)f(\rho)=\int_0^1 K_n'(\rho,s;\tau)f(s)\d s. \]
From Lemma \ref{lem:boundsG'} we have the bounds
\[ |K_n'(\rho,s;\tau)|\lesssim \rho^{-1}(1-\rho)^{-\frac12}s(1-s)^{-\frac12}
\langle \tau-\log(1-\rho)+\log(1-s)\rangle^{-2} \]
and thus, by the change of variables $\rho=1-e^{-x}$, $s=1-e^{-y}$, we infer
\begin{align*} &|[S_n'(\tau)f](1-e^{-x})|  \lesssim e^{x/2}(1-e^{-x})^{-1}\int_0^\infty \langle \tau+x-y\rangle^{-2}
|f(1-e^{-y})|(1-e^{-y})e^{-y/2}\d y .
\end{align*}
Consequently, Young's inequality yields
\begin{align*}
\|S_n'(\tau)f\|_{L^2(\B^3)}&\simeq \left \|[S_n'(\tau)f](1-e^{-|\cdot|})(1-e^{-|\cdot|})e^{-|\cdot|/2} \right \|_{L^2(\R_+)} \\
&\lesssim \left \|\langle \tau+\cdot\rangle^{-2}\right \|_{L^1(\R)}
\left \|f(1-e^{-|\cdot|})(1-e^{-|\cdot|})e^{-|\cdot|/2}\right \|_{L^2(\R_+)} \\
&\simeq \|f\|_{L^2(\B^3)}.
\end{align*}
The bounds for $\tilde S_n(\tau)$ follow in the same way by noting that
\begin{align*} 
\left |\int_\R e^{\I\omega\tau}\rho^{-1}\tilde G_n(\rho,s;\I\omega)\d\omega\right |&\lesssim 
\rho^{-1}s(1-s)^{-\frac12}\langle\tau+\log(1-s)\rangle^{-2} \\
&\lesssim \rho^{-1}(1-\rho)^{-\frac12}s(1-s)^{-\frac12}
\langle \tau-\log(1-\rho)+\log(1-s)\rangle^{-2},
\end{align*}
see Proposition \ref{prop:Gn}.
\end{proof}

\subsection{The operators $\dot S_n'(\tau)$}

As before, we need to trade a derivative in $s$ for decay in $\lambda$ in order to control
the term $\lambda\tilde f_1$ in Eq.~\eqref{eq:SG'1}.
Thus, for $\tau\geq 0$, $\rho\in (0,1)$, $f\in C^1([0,1])$, and $n\in \{1,2,3,4\}$, we define the operators
\[ \dot S_{n,\epsilon}'(\tau)f(\rho):=\frac{1}{2\pi \I}\lim_{N\to\infty}\int_{\epsilon-\I N}^{\epsilon+\I N}
\lambda e^{\lambda\tau}\int_0^1 G_n'(\rho,s;\lambda)f(s)\d s\d\lambda. \]

\begin{lemma}
\label{lem:dotSn'}
Set
\[ \dot S_n'(\tau)f(\rho)=\lim_{\epsilon\to 0+}\dot S_{n,\epsilon}'(\tau)f(\rho). \]
Then we have
\[ \|\dot S_n'(\tau)f\|_{L^2(\B^3)}\lesssim \|f\|_{H^1(\B^3)} \]
for all $\tau\geq 0$, $f\in C^1([0,1])$, and $n\in \{1,2,3,4\}$.
\end{lemma}

\begin{proof}
The $s$-dependent part of $G_1'(\rho,s;\lambda)$ reads
\[ 1_{\R_+}(s-\rho)s(1-s)^{-\frac12+\lambda}\gamma_1'(\rho,s;\lambda). \]
An integration by parts yields
\begin{align*}
 \int_0^1 &1_{\R_+}(s-\rho)s(1-s)^{-\frac12+\lambda}\gamma_1'(\rho,s;\lambda)f(s)\d s \\
 &=-\frac{2}{1+2\lambda}\int_\rho^1 \partial_s (1-s)^{\frac12+\lambda}s\gamma_1'(\rho,s;\lambda)f(s)\d s \\
 &=\frac{2}{1+2\lambda}(1-\rho)^{\frac12+\lambda}\rho\gamma_1'(\rho,\rho;\lambda)f(\rho) \\
 &\quad +\frac{2}{1+2\lambda}\int_0^1 1_{\R_+}(s-\rho)(1-s)^{\frac12+\lambda}\partial_s [s\gamma_1'(\rho,s;\lambda)
 f(s)]\d s.
 \end{align*}
 We have 
 \[ s\partial_s \gamma_1'(\rho,s;\lambda)=\O(\rho^0s^0)\Oo(\langle\omega\rangle^{-1})+
 \O(\rho^0(1-\rho)^0 s^0(1-s)^0\langle\omega\rangle^{-2}) \]
 and since $\|f\|_{L^2(0,1)}\lesssim \|f\|_{H^1(\B^3)}$, the integral term leads to an operator
 that is bounded from $H^1(\B^3)$ to $L^2(\B^3)$.
By H\"older's inequality, the same is true for the boundary term, cf.~Proposition \ref{prop:dotTn}.
The other cases are handled similarly.
\end{proof}

\subsection{Energy bounds}
Finally, we can conclude the desired energy bound.

\begin{lemma}
\label{lem:energy}
For the semigroup $\mb S$ generated by $\mb L_0+\mb L'$, see Proposition \ref{prop:SG},
we have the bound
\[ \|\mb S(\tau)(\mb I-\mb P)\mb f\|_{\mc H}\lesssim \|(\mb I-\mb P)\mb f\|_{\mc H} \]
for all $\tau\geq 0$ and all $\mb f\in \mc H$.
\end{lemma}

\begin{proof}
Let $\tilde{\mb f}:=(\mb I-\mb P)\mb f\in C^2\times C^1([0,1])$.
From the representation
\[  [\mb S(\tau)\tilde{\mb f}]_1=[\mb S_0(\tau)\tilde{\mb f}]_1
+\sum_{n=1}^6 \left [ T_n(\tau)(|\cdot|\tilde f_1'+\tfrac32 \tilde f_1+\tilde f_2)
+\dot T_n(\tau)\tilde f_1 \right ], \]
Propositions \ref{prop:gen}, \ref{prop:dotTn}, Lemma \ref{lem:Tn}, and H\"older's inequality we infer
the bound
\[ \|[\mb S(\tau)\tilde{\mb f}]_1\|_{L^2(\B^3)}\lesssim \|\tilde{\mb f}\|_{\mc H}. \]
Furthermore, from Eq.~\eqref{eq:SG'1} and Lemmas \ref{lem:SL2}, \ref{lem:dotSn'} we obtain
\[ \|[\mb S(\tau)\tilde{\mb f}]_1\|_{\dot H^1(\B^3)}\lesssim \|\tilde{\mb f}\|_{\mc H}. \]
By definition, see Eq.~\eqref{eq:psitopsi12}, the second component of 
$\mb S(\tau)\tilde{\mb f}$ is given by
\[ [\mb S(\tau)\tilde{\mb f}]_2(\rho)=[\partial_\tau+\rho\partial_\rho+\tfrac12]
[\mb S(\tau)\tilde{\mb f}]_1(\rho) \]
and since a $\tau$-derivative produces a factor of $\lambda$, the operators $\dot T_n(\tau)$ and
$\partial_\tau \dot T_n(\tau)$ are comparable to $S_n'(\tau)$ and $\dot S_n'(\tau)$, respectively.
Consequently, we obtain from
Lemmas \ref{lem:SL2}, \ref{lem:dotSn'} the bound
\[ \|[\mb S(\tau)\tilde{\mb f}]_2\|_{L^2(\B^3)}\lesssim \|\tilde{\mb f}\|_{\mc H}. \]
and the claim follows by a density argument.
\end{proof}

\section{The nonlinear problem}

\noindent Now we turn to the nonlinear problem
\[ \Phi(\tau)=\mb S(\tau)\mb u+\int_0^\tau \mb S(\tau-\sigma)\mb N(\Phi(\sigma))\d\sigma. \]
Recall that
\[ \mb N(\mb u)=\left (\begin{array}{c} 0 \\ N(u_1) \end{array} \right ) \]
where
\[ N(u_1)=10c_3^3 u_1^2+10 c_3^2 u_1^3+5c_3 u_1^4+u_1^5. \]
The control of the nonlinearity will be based on the following simple bounds combined
with Strichartz estimates.

\begin{lemma}
\label{lem:N}
We have the bounds
\begin{align*} 
\|\mb N(\mb u)\|_{\mc H}&\lesssim \|u_1\|_{L^{10}(\B^3)}^2+\|u_1\|_{L^{10}(\B^3)}^5 \\
\|\mb N(\mb u)-\mb N(\mb v)\|_{\mc H}&\lesssim \Big (\|u_1\|_{L^{10}(\B^3)}
+\|u_1\|_{L^{10}(\B^3)}^4 \\
&\quad +\|v_1\|_{L^{10}(\B^3)}+\|v_1\|_{L^{10}(\B^3)}^4 \Big )\|u_1-v_1\|_{L^{10}(\B^3)}
\end{align*}
for all $\mb u, \mb v \in \mc H$.
\end{lemma}

\begin{proof}
The estimates are a simple consequence of H\"older's inequality and elementary
identities such as $u_1^2-v_1^2=(u_1+v_1)(u_1-v_1)$.
\end{proof}

\subsection{The modified equation}
Since the linear evolution has an unstable direction, we have to first modify the 
equation in order to construct a global (in $\tau$) solution.
In a second step we then show how to remove this modification.

For given initial data $\mb u \in \mc H$ we consider the map
\begin{align*} \mb K_\mb u(\Phi)(\tau):&=\mb S(\tau)\mb u
+\int_0^\tau \mb S(\tau-\sigma)\mb N(\Phi(\sigma))\d\sigma-e^\tau \mb C(\Phi,\mb u) \\
&=\mb S(\tau)[\mb u-\mb C(\Phi,\mb u)]
+\int_0^\tau \mb S(\tau-\sigma)\mb N(\Phi(\sigma))\d\sigma
\end{align*}
where
\[ \mb C(\Phi,\mb u):=\mb P\left [\mb u+\int_0^\infty e^{-\sigma}
\mb N(\Phi(\sigma))\d\sigma \right ]. \] 
At this stage these definitions are purely formal.
Now we introduce suitable function spaces and prove mapping properties of $\mb K_\mb u$.
For a function $\Phi(\tau)(\rho)=(\phi_1(\tau,\rho),\phi_2(\tau,\rho))$ we define
\[ \|\Phi\|_{\mc X}^2:=\|\Phi\|_{L^\infty(\R_+)\mc H}^2+\|\phi_1\|_{L^2(\R_+)L^\infty(\B^3)}^2 \]
and introduce the Banach space
\[ \mc X:=\left \{\Phi \in C([0,\infty),\mc H): \phi_1\in L^2(\R_+,L^\infty(\B^3)),
\|\Phi\|_{\mc X}<\infty \right \}. \]
Furthermore, we set
\[ \mc X_\delta:=\left \{\Phi\in \mc X: \|\Phi\|_{\mc X}\leq \delta \right \}. \]

\begin{lemma}
\label{lem:Kball}
There exists a $c>0$ such that the following holds.
If $\|\mb u\|_{\mc H}\leq \frac{\delta}{c}$ and $\Phi \in \mc X_\delta$
then $\mb K_{\mb u}(\Phi)\in \mc X_\delta$, provided $\delta>0$ is sufficiently small.
\end{lemma}

\begin{proof}
By Lemmas \ref{lem:energy} and \ref{lem:N} we have
\begin{align*}
\|(\mb I-\mb P)\mb K_\mb u(\Phi)(\tau)\|_{\mc H}&\lesssim \|\mb u\|_{\mc H}+\int_0^\tau \|\mb N(\Phi(\sigma))\|_{\mc H}\d\sigma \\
&\lesssim \tfrac{\delta}{c}+\int_0^\tau \left (
\|\phi_1(\sigma,\cdot)\|_{L^{10}(\B^3)}^2+\|\phi_1(\sigma,\cdot)\|_{L^{10}(\B^3)}^5 \right ) \d\sigma \\
&\lesssim \tfrac{\delta}{c}+\|\phi_1\|_{L^2(\R_+)L^\infty(\B^3)}^2 
+\|\phi_1\|_{L^5(\R_+)L^{10}(\R_+)}^5\\
&\lesssim \tfrac{\delta}{c}+\delta^2
+\left (\|\phi_1\|_{L^\infty(\R_+)L^6(\B^3)}^\theta
\|\phi_1\|_{L^2(\R_+)L^\infty(\B^3)}^{1-\theta}\right )^5 \\
&\lesssim \tfrac{\delta}{c}+\delta^2
+\left (\|\Phi\|_{L^\infty(\R_+)\mc H}^\theta
\|\phi_1\|_{L^2(\R_+)L^\infty(\B^3)}^{1-\theta}\right )^5 \\
&\lesssim \tfrac{\delta}{c}+\delta^2
+\delta^5. 
\end{align*}
Furthermore, the Strichartz estimates from Theorem \ref{thm:strich} imply
\begin{align*}
\|[(\mb I-\mb P)\mb K_\mb u(\Phi)]_1\|_{L^2(\R_+)L^\infty(\B^3)}
&\lesssim \|\mb u\|_{\mc H}+\int_0^\infty \|\mb N(\Phi(\tau))\|_{\mc H}\d\tau \\
&\lesssim \tfrac{\delta}{c}+\int_0^\infty \|\phi_1(\tau,\cdot)\|_{L^{10}(\B^3)}^2 \d\tau \\
&\quad +\int_0^\infty \|\phi_1(\tau,\cdot)\|_{L^{10}(\B^3)}^5 \d\tau \\
&\lesssim \tfrac{\delta}{c}+\delta^2+\delta^5.
\end{align*}
Next, we consider $\mb P\mb K_\mb u(\Phi)(\tau)$ which is given by
\[ \mb P\mb K_\mb u(\Phi)(\tau)=-\int_\tau^\infty e^{\tau-\sigma}\mb P\mb N(\Phi(\sigma))\d\sigma. \]
Recall that $\rg \mb P=\langle \mb g\rangle$ and thus, by Riesz' representation theorem
there exists a $\mb g^*\in \mc H$ such that
\[ \mb P \mb f=(\mb f|\mb g^*)_{\mc H}\,\mb g \]
for all $\mb f\in \mc H$.
Consequently, we obtain
\[ \mb P\mb K_\mb u(\Phi)(\tau)=-\mb g\int_\tau^\infty e^{\tau-\sigma}\left (
\mb N(\Phi(\sigma))|\mb g^*\right )_{\mc H}\,\d\sigma. \]
This yields
\begin{align*} \|\mb P\mb K_{\mb u}(\Phi)(\tau)\|_{\mc H}
&\lesssim \int_\tau^\infty e^{\tau-\sigma}
\left |\left (
\mb N(\Phi(\sigma))|\mb g^*\right )_{\mc H}\right |\d\sigma \\
&\lesssim \int_{\tau}^\infty \|\mb N(\Phi(\sigma))\|_{\mc H}\,\d\sigma \\
&\lesssim \int_0^\infty \left (\|\phi_1(\sigma,\cdot)\|_{L^\infty(\B^3)}^2
+\|\phi_1(\sigma,\cdot)\|_{L^{10}(\B^3)}^5 \right )
\d\sigma \\
&\lesssim \delta^2+\delta^5.
\end{align*}
Finally, we obtain
\begin{align*}
\|[\mb P\mb K_{\mb u}(\Phi)(\tau)]_1\|_{L^\infty(\B^3)}
&\lesssim \int_\tau^\infty e^{\tau-\sigma}\|\mb N(\Phi(\sigma))\|_{\mc H}\,\d\sigma \\
&=\int_\R 1_{(-\infty,0]}(\tau-\sigma)e^{\tau-\sigma}1_{[0,\infty)}(\sigma)\|\mb N(\Phi(\sigma))\|_{\mc H}\,
\d\sigma
\end{align*}
and Young's inequality yields
\begin{align*}
\|[\mb P\mb K_\mb u(\Phi)]_1\|_{L^2(\R_+)L^\infty(\B^3)}&\lesssim \left \|1_{(-\infty,0]}
e^{|\cdot|}\right \|_{L^2(\R)} \int_0^\infty \|\mb N(\Phi(\sigma))\|_{\mc H}\,\d\sigma \\
&\lesssim \delta^2+\delta^5.
\end{align*}
In summary, we infer $\|\mb K_{\mb u}(\Phi)\|_{\mc X}\lesssim \frac{\delta}{c}+\delta^2
+\delta^5$, which implies the claim.
\end{proof}

\begin{lemma}
\label{lem:Kcontr}
Let $\delta>0$ be sufficiently small and $\mb u\in \mc H$.
Then we have the estimate
\[ \|\mb K_\mb u(\Phi)-\mb K_\mb u(\Psi)\|_{\mc X}\leq \tfrac12 \|\Phi-\Psi\|_{\mc X} \]
for all $\Phi,\Psi \in \mc X_\delta$.
\end{lemma}

\begin{proof}
By Lemmas \ref{lem:energy} and \ref{lem:N} we find
\begin{align*}
&\left \|(\mb I-\mb P)[\mb K_\mb u(\Phi)(\tau)-\mb K_\mb u(\Psi)(\tau)]\right \|_{\mc H}  \\
 &\lesssim \int_0^\tau \|\mb N(\Phi(\sigma))-\mb N(\Psi(\sigma))\|_{\mc H}\d\sigma \\
 &\lesssim \int_0^\tau \left (\|\phi_1(\sigma)\|_{L^{10}(\B^3)}
 +\|\phi_1(\sigma)\|_{L^{10}(\B^3)}^4+\|\psi_1(\sigma)\|_{L^{10}(\B^3)}
 +\|\psi_1(\sigma)\|_{L^{10}(\B^3)}^4\right ) \\
 &\quad \times \|\phi_1(\sigma)-\psi_1(\sigma)\|_{L^{10}(\B^3)}
 \d\sigma \\
 &\lesssim \|\phi_1\|_{L^2(\R_+)L^\infty(\B^3)}\|\phi_1-\psi_1\|_{L^2(\R_+)L^\infty(\B^3)}
 +\|\phi_1\|_{L^5(\R_+)L^{10}(\B^3)}^4 \|\phi_1-\psi_1\|_{L^5(\R_+)L^{10}(\B^3)} \\\
 &\quad +\|\psi_1\|_{L^2(\R_+)L^\infty(\B^3)}\|\phi_1-\psi_1\|_{L^2(\R_+)L^\infty(\B^3)}
 +\|\psi_1\|_{L^5(\R_+)L^{10}(\B^3)}^4 \|\phi_1-\psi_1\|_{L^5(\R_+)L^{10}(\B^3)} \\
 &\lesssim \delta \|\Phi-\Psi\|_{\mc X}.
 \end{align*}
 Furthermore, Theorem \ref{thm:strich} and Lemma \ref{lem:energy} yield
 \begin{align*}
 &\left \|[(\mb I-\mb P)\mb K_\mb u(\Phi)-(\mb I-\mb P)\mb K_\mb u(\Psi)]_1
 \right \|_{L^2(\R_+)L^\infty(\B^3)} \\
 &\lesssim \int_0^\infty \|\mb N(\Phi(\tau))-\mb N(\Psi(\tau))\|_{\mc H}\d \tau \\
 &\lesssim \delta \|\Phi-\Psi\|_{\mc X}
 \end{align*}
 by the same logic as above.
 
 On the unstable subspace we use $\mb P\mb f=(\mb f|\mb g^*)_\mc H\mb g$ as in the proof
 of Lemma \ref{lem:Kball} to obtain
 \begin{align*}
 \|\mb P\mb K_\mb u(\Phi)(\tau)-\mb P \mb K_\mb u(\Psi)(\tau)\|_{\mc H}
 &\lesssim \int_\tau^\infty e^{\tau-\sigma} \left | \left (\mb N(\Phi(\sigma))-\mb N(\Psi(\sigma))
 |\mb g^*
 \right )_{\mc H}\right |\d\sigma \\
 &\lesssim \int_\tau^\infty \|\mb N(\Phi(\sigma))-\mb N(\Psi(\sigma))\|_{\mc H}\d\sigma \\
  &\lesssim \delta \|\Phi-\Psi\|_{\mc X}.
 \end{align*}
 Finally, 
 \begin{align*}
 \|[\mb P\mb K_\mb u(\Phi)(\tau)-\mb P\mb K_\mb u(\Psi)(\tau)]_1\|_{L^\infty(\B^3)}
 &\lesssim \int_\tau^\infty e^{\tau-\sigma}\|\mb N(\Phi(\sigma))-\mb N(\Psi(\sigma))\|_{\mc H}
 \d\sigma
 \end{align*}
 and as in the proof of Lemma \ref{lem:Kball}, we infer from Young's inequality the estimate
 \[ \|[\mb P\mb K_\mb u(\Phi)-\mb P\mb K_\mb u(\Psi)]_1\|_{L^2(\R_+)L^\infty(\B^3)}
 \lesssim \delta \|\Phi-\Psi\|_{\mc X}. \]
\end{proof}

\begin{corollary}
\label{cor:modEq}
There exist $c,\delta>0$ such that, for every $\mb u$ with $\|\mb u\|_{\mc H}\leq \frac{\delta}{c}$,
there exists a unique $\Phi\in \mc X_\delta$ satisfying $\Phi=\mb K_{\mb u}(\Phi)$.
\end{corollary}

\begin{proof}
This is a consequence of Lemmas \ref{lem:Kball}, \ref{lem:Kcontr}, and the Banach
fixed point theorem.
\end{proof}

\subsection{Variation of blowup time}
In the final step we show that choosing the correct blowup time makes the correction
term $\mb C(\Phi,\mb u)$ go away. 
As a consequence, we obtain a solution to the original equation.

Recall that by Eqs.~\eqref{eq:utopsi}, \eqref{eq:psitopsi12}, and \eqref{eq:ansatzphi}, the
initial data $\Phi(0)=(\phi_1(0,\cdot),\phi_2(0,\cdot))$ we prescribe are of the form
\begin{align*} 
\phi_1(0,\rho)&=\psi_1(0,\rho)-c_3=T^\frac12 u(0,T\rho)-c_3 =T^\frac12 f(T\rho)-c_3 \\
\phi_2(0,\rho)&=\psi_2(0,\rho)-\tfrac12c_3=T^{\frac32}\partial_0 u(0,T\rho)-\tfrac12 c_3
=T^{\frac32}g(T\rho)-\tfrac12 c_3.
\end{align*}
In the formulation of Theorem \ref{thm:main} we measure the size of the initial data
relative to the ODE blowup solution $u^1$.
According to Eq.~\eqref{eq:utopsi}, the latter corresponds to
\[ \psi^1(\tau,\rho):=T^\frac12 e^{-\frac12\tau} 
u^1(T-Te^{-\tau},Te^{-\tau}\rho)=c_3 T^\frac12 e^{-\frac12\tau}(1-T+Te^{-\tau})^{-\frac12}. \]
In view of Eq.~\eqref{eq:psitopsi12}, we set $\psi_1^1:=\psi^1$ and
\[ \psi^1_2(\tau,\rho):=[\partial_\tau+\rho\partial_\rho+\tfrac12]\psi^1(\tau,\rho)
= \tfrac12 c_3 T^\frac32 e^{-\frac32\tau}(1-T+Te^{-\tau})^{-\frac32}. \]
Consequently, this blowup solution has initial data
\[ \psi_1^1(0,\rho)=c_3 T^\frac12,\qquad \psi_2^1(0,\rho)=\tfrac12 c_3 T^\frac32. \]
We therefore rewrite our initial data as
\begin{align*} \Phi(0)(\rho)&=(T^\frac12 f(T\rho),T^\frac32g(T\rho))-(c_3 T^\frac12,\tfrac12 c_3 T^\frac32)
+(c_3 T^\frac12,\tfrac12 c_3 T^\frac32)-(c_3,\tfrac12 c_3)  \\
&=\mb U(T,(f-c_3,g-\tfrac12 c_3))(\rho)
\end{align*}
where 
\[ \mb U(T,\mb v)(\rho):=(T^\frac12 v_1(T\rho), T^\frac32 v_2(T\rho))
+(c_3 T^\frac12,\tfrac12 c_3 T^\frac32)-(c_3,\tfrac12 c_3).  \]
It is not hard to see that, for $\delta >0$ small enough, the map 
\[ \mb U: [1-\delta,1+\delta]\times H^1(\B^3_{1+\delta})\times L^2(\B^3_{1+\delta})
\to H^1(\B^3)\times L^2(\B^3) \] is continuous and $\mb U(1,\mb 0)=\mb 0$.
Furthermore, one easily checks that 
\[ \|\mb U(T,\mb v)\|_{H^1\times L^2(\B^3)}\lesssim \|\mb v\|_{H^1\times L^2(\B^3_{1+\delta})}+|T-1| \]
for all $T\in [1-\delta,1+\delta]$.

\begin{lemma}
\label{lem:blowupT}
There exist $M\geq 1$ and $\delta>0$ such that the following holds.
For any $\mb v \in H^1\times L^2(\B^3_{1+\delta})$ satisfying
$\|\mb v\|_{H^1\times L^2(\B^3_{1+\delta})}\leq \tfrac{\delta}{M}$, there exists
a $T^*\in [1-\delta,1+\delta]$ and a function $\Phi\in \mc X_{\delta}$ 
which satisfies $\Phi=\mb K_{\mb U(T^*,\mb v)}(\Phi)$ and 
$\mb C(\Phi,\mb U(T^*,\mb v))=\mb 0$.
\end{lemma}

\begin{proof}
Note that
\[ \partial_T (c_3T^\frac12,\tfrac12 c_3 T^\frac32)|_{T=1}=(\tfrac12 c_3,\tfrac34 c_3)
=\tfrac14 c_3 \mb g \]
with $\mb g$ from Proposition \ref{prop:SG}.
Thus, we may write
\[ \mb U(T,\mb v)(\rho)=(T^\frac12 v_1(T\rho), T^\frac32 v_2(T\rho))+\tfrac14 c_3 (T-1)\mb g 
+(T-1)^2 \mb f_T \]
where $\|\mb f_T\|_{\mc H}\lesssim 1$ for all $T\in [\frac12,\frac32]$.
This yields
\[ (\mb U(T,\mb v)|\mb g)_{\mc H}=O(\tfrac{\delta}{M}T^0)+\tfrac14 c_3\|\mb g\|_{\mc H}^2(T-1)
+O(\delta^2 T^0) \]
for all $T\in [1-\delta,1+\delta]$, $\delta\in [0,\frac12]$ and $M\geq 1$.
Now let $\Phi_T\in \mc X_\delta$ be the fixed point of $\mb K_{\mb U(T,\mb v)}$
which, by Corollary \ref{cor:modEq}, exists for any $T\in [1-\delta,1+\delta]$, provided
$\delta>0$ is sufficiently small and $M$ is sufficiently large.
We have
\[ \mb C(\Phi_T,\mb U(T,\mb v))=\mb P\mb U(T,\mb v)+\mb P\int_0^\infty e^{-\sigma}
\mb N(\Phi_T(\sigma))\d\sigma \]
and by Lemma \ref{lem:N} we obtain
\[ \int_0^\infty e^{-\sigma}\|\mb N(\Phi_T(\sigma))\|_{\mc H}\d\sigma\lesssim \delta^2, \]
cf.~the proof of Lemma \ref{lem:Kball}.
Consequently,
\[ (\mb C(\Phi_T,\mb U(T,\mb v))|\mb g)_{\mc H}=\tfrac14 c_3 \|\mb g\|_{\mc H}^2(T-1)+
O(\tfrac{\delta}{M}T^0)+O(\delta^2T^0), \]
and the $O$-terms are continuous functions of $T$.
Since $\mb C(\Phi_T,\mb U(T,\mb v))\in \langle \mb g\rangle$, we see that
$\mb C(\Phi_T,\mb U(T,\mb v))=\mb 0$ is equivalent to $T-1=F(T)$, where $F$ is 
a continuous function on $[1-\delta,1+\delta]$ which satisfies $|F(T)|\lesssim \frac{\delta}{M}+\delta^2$. 
Thus, by choosing $M\geq 1$ large enough and then $\delta>0$ small enough, we see that $1+F$ is a continuous self-map of the
interval $[1-\delta,1+\delta]$ which necessarily has a fixed point $T^*$.
\end{proof}

\begin{remark}
It is not difficult to see that $T^*$ from Lemma \ref{lem:blowupT} is 
unique in $[1-\delta,1+\delta]$. This follows from the fact that $T^*$ can be obtained
as a fixed point of a contraction mapping.
\end{remark}

\subsection{Proof of Theorem \ref{thm:main}}
Let $M\geq 1$ be sufficiently large and choose $\delta>0$ sufficiently small.
For given $\mb v=(f,g)-u^1[0] \in H^1\times L^2(\B^3_{1+\delta})$ with
\[ \|(f,g)-u^1[0]\|_{H^1\times L^2(\B^3_{1+\delta})}\leq \tfrac{\delta}{M}, \]
let $\Phi\in \mc X_\delta$ be the associated solution from Lemma \ref{lem:blowupT}.
Furthermore, let $T=T^*$ be the corresponding blowup time from Lemma \ref{lem:blowupT}.
Then, by Eqs.~\eqref{eq:ansatzphi}, \eqref{eq:psitopsi12}, and \eqref{eq:utopsi}, we have
\begin{align*}
\delta^2 &\geq \|\phi_1\|_{L^2(\R_+)L^\infty(\B^3)}^2
=\int_0^\infty \|\phi_1(\tau)\|_{L^\infty(\B^3)}^2\d\tau
=\int_0^\infty \|\psi(\tau,\cdot)-c_3\|_{L^\infty(\B^3)}^2\d\tau \\
&=\int_0^T \|\psi(-\log(T-t)+\log T,\cdot)-c_3\|_{L^\infty(\B^3)}^2 \frac{\d t}{T-t} \\
&=\int_0^T \|\psi(-\log(T-t)+\log T,\tfrac{\cdot}{T-t})-c_3\|_{L^\infty(\B^3_{T-t})}^2
\frac{\d t}{T-t} \\
&=\int_0^T (T-t)\|(T-t)^{-\frac12}
\psi(-\log(T-t)+\log T,\tfrac{\cdot}{T-t})-c_3(T-t)^{-\frac12}\|_{L^\infty(\B^3_{T-t})}^2
\frac{\d t}{T-t} \\
&\simeq \int_0^T \frac{\|u(t,\cdot)-u^T(t,\cdot)\|_{L^\infty(\B^3_{T-t})}^2}
{\|u^T(t,\cdot)\|_{L^\infty(\B^3_{T-t})}^2}
\frac{\d t}{T-t}.
\end{align*}

\bibliography{strich}
\bibliographystyle{plain}

\end{document}